\documentclass{amsart} 
\usepackage{amssymb,amsfonts, enumerate}
\usepackage{mathrsfs}
\usepackage[mathcal]{eucal}
\usepackage{hyperref}

\theoremstyle{plain}

\newtheorem{theorem}{Theorem}[section]
\newtheorem*{theorem1}{Theorem 1}
\newtheorem*{theorem2}{Theorem 2}

\newtheorem{proposition}[theorem]{Proposition}
\newtheorem{lemma}[theorem]{Lemma}
\newtheorem{corollary}[theorem]{Corollary}

\theoremstyle{definition}

\newtheorem{definition}[theorem]{Definition}

\newtheorem{convention}[theorem]{Convention}
\newtheorem{notation}[theorem]{Notation}

\newtheorem{remark}[theorem]{Remark}

\theoremstyle{remark}
\newtheorem*{claim}{Claim}

\DeclareMathOperator{\Mod}{Mod}
\DeclareMathOperator{\Sent}{Sent}
\DeclareMathOperator{\Str}{Str}

\DeclareMathOperator{\To}{\Rightarrow}

\DeclareMathOperator{\density}{density}

\DeclareMathOperator{\cl}{CL}
\DeclareMathOperator{\Discrete}{Discrete}
\DeclareMathOperator{\lukimp}{\rightarrow_{L}}

\renewcommand{\phi}{\varphi}
\newcommand{\nmodels}{\nvDash}

\begin{document}

\title[Omitting uncountable types]
{Omitting uncountable types, \\
and the strength of $[0,1]$-valued logics}

\author[X. Caicedo]{Xavier Caicedo}

\address{Department of Mathematics\\
  Universidad de los Andes\\
  Apartado Aereo 4976\\
  Bogot\'{a}, Colombia}

\email{xcaicedo@uniandes.edu.co}

\author[J. N. Iovino]{Jos\'e N. Iovino}

\address{Department of Mathematics\\
  The University of Texas at San Antonio\\
  One UTSA Circle\\
  San Antonio, TX 78249-0664\\
  USA}

\email{iovino@math.utsa.edu}

\date{\today}
\subjclass[2000]{03C95, 
03C90,
03B50,
03B52,
54E52.
 }

\begin{abstract}
We study $[0,1]$-valued logics that are closed under the  \L ukasiewicz-Pavelka connectives; our primary examples are the the continuous logic framework of Ben Yaacov and Usvyatsov~\cite{Ben-Yaacov-Usvyatsov:2010} and the {\L}ukasziewicz-Pavelka logic itself. The main result of the paper is a characterization of these logics in terms of a model-theoretic property, namely, an extension of the omitting types theorem to uncountable languages. 
\end{abstract}

\maketitle

\section*{Introduction}

In this paper we deal with the model theory $[0,1]$-valued logics. We focus on logics that are closed under the \L ukasiewicz-Pavelka connectives. Our primary examples are the continuous logic framework of Ben Yaacov and Usvyatsov~\cite{Ben-Yaacov-Usvyatsov:2010} and the \L ukasiewicz-Pavelka logic itself~\cite{Pavelka:1979I,Pavelka:1979II,Pavelka:1979III} (see also Section~5.4 of~\cite{Hajek:1998}). The main result of the paper is a characterization of these logics in terms of a model-theoretic property, namely, an extension of the omitting types theorem to uncountable languages. 

If $\lambda$ is an uncountable cardinal and $T$ is a theory of cardinality~$\le\lambda$, we will say that a partial type $\Sigma(x)$ in a  logic $\mathcal{L}$ is \emph{$\lambda$-principal} if there exists a set of formulas~$\Phi(x)$ of cardinality~$<\lambda$  such that $T\cup\Phi(x)$ is consistent and $T\cup\Phi(x) \models\Sigma(x)$. 
 
 A logic $\mathcal{L}$ satisfies the \emph{$\lambda$-Omitting Types Property} if 
  whenever $T$ is a consistent theory of cardinality $\leq \lambda$  and  $\{\Sigma_{j}(x)\}_{j < \lambda}$ is a set of types that are not $\lambda$-principal over $T$ there is a model of $T$ that omits each $\Sigma_{j}(x)$.
 
In the paper we work within a  logic $\mathcal L$ whose semantics is given by the class of continuous metric structures. We prove that the uncountable Omitting Types Property defined above characterizes $\mathcal L$. Then, analogous characterizations of \L ukasiewicz-Pavelka logic and the continuous logic framework of ~\cite{Ben-Yaacov-Usvyatsov:2010}  follow, by restricting $\mathcal L$ to specific classes of structures.

The sentences of $\mathcal{L}$ are $[0,1]$-valued. The connectives of $\mathcal L$ are the \L ukasiewicz implication ($\phi\to\psi=\min\{1-\phi+\psi,1\}$), and the Pavelka rational constants, i.e., for each rational $r$ in the closed interval $[0,1]$ a constant connective with value $r$.  The quantifiers are $\forall$ and $\exists$ (only one of them is needed). 

We observe that the restriction of $\mathcal L$ to the class of 1-Lipschitz structures is predicate \L ukasiewicz-Pavelka logic, and its restriction to the class of complete structures yields the continuous logic framework of~\cite{Ben-Yaacov-Usvyatsov:2010}.
  
 In the first part of the paper we prove the following result:
 
 \begin{theorem1}\label{T:omitting types characterization}
 $\mathcal{L}$ satisfies the $\lambda$-Omitting Types Property for every uncountable cardinal~$\lambda$. 
 \end{theorem1} 
 
In the second part we show that this property characterizes $\mathcal{L}$:

\begin{theorem2}
Let $\mathcal{L}'$ be a $[0,1]$-valued logic  that extends $\mathcal{L}$ and satisfies the following properties:
\begin{itemize}
\item
The $\lambda$-Omitting Types Property for every uncountable cardinal~$\lambda$,
\item
Closure under the of \L ukasiewicz-Pavelka connectives \emph{(see below)} and the existential quantifier,
\item
Every continuous metric structure is logically equivalent in $\mathcal{L}'$ to its metric completion.
\end{itemize}
Then every sentence in $\mathcal{L}'$ is is a uniform limit of sentences in $\mathcal{L}$.
\end{theorem2}

Theorem~2 is the main result of the paper, and generalizes a characterization of first-order logic due to Lindstr\"om~\cite{Lindstrom:1978}. By restricting Theorem~2 to the class of 1-Lipschitz structures
we obtain a characterization of \L ukasiewicz-Pavelka logic, and by restricting it to the class of complete structures we obtain an analogous characterization of continuous logic. See Corollary~\ref{C: maximality}.

The latter case uses a form of the $\lambda$-Omitting Types Property that asserts that the type-omitting structure is complete. This version requires a stronger notion of type principality, but it is a direct consequence of the $\lambda$-Omitting Types Property of $\mathcal L$.

  The logic $\mathcal L$ is an instance of the general concept of continuous logic introduced by Chang and Keisler in~\cite{Chang-Keisler:1966}. Nevertheless, in order to be consistent with the use of the term ``continuous logic'' in recent literature, in the paper we refer to $\mathcal L$ as ``basic continuous logic''. 

Our proof of the $\lambda$-Omitting Types Property is based on a general version of the Baire Category Theorem (Proposition~\ref{Proposition:UncountableVersionOfBaireCategoryTheorem}). The proof covers at once  the uncountable case discussed above and the case $\lambda=\omega$. (See Theorem~\ref{Theorem:GeneralOmittingTypesTheorem}.) The countable is not new; omitting types theorems for $[0,1]$-valued logics over countable languages have been proved by Murinov{\'a}-Nov\'{a}k~\cite {Murinova-Novak:2006} (for  {\L}ukasziewicz-Pavelka logic) and by Henson~\cite{Ben-Yaacov-Usvyatsov:2007, Ben-Yaacov-Berenstein-Henson-Usvyatsov:2008} (for complete metric structures).  
 
Our approach is topological. The usefulness of topological methods to study model-theoretic properties of abstract logics is at the heart of several papers by the first author (see, for example,~\cite{Caicedo:1993,Caicedo:1995,Caicedo:1999}). The topological approach followed in those papers is particularly well-suited for settings such as those considered here, where the logics at hand do not have negation in the classical sense. In such settings the assumption that the space of structures is  topologically  regular serves as the ``correct'' replacement of the classical negation.  (Indeed, the collaboration between both authors originated with the realization that the concept of regular logic, which was isolated by the first author, is equivalent to the concept of logics with weak negation that was introduced by the second author in~\cite{Iovino:2001}.) Utilizing these ideas, the first author has proved topological versions of Lindstr\"om's First Theorem for first-order logic~\cite {Caicedo:201?}.

The paper is self-contained; no previous familiarity with abstract model theory, with \L ukasiewicz logic, or with continuous logic is presumed. The basic definitions are given in Section~\ref{Section:StructuresAndLogics}. Section~\ref{Section:TheLogicalTopology} introduces the topological preliminaries. Section~\ref{Section:A General Omitting Types Theorem} is devoted to the proof of the $\lambda$-Omitting Types Property, and Section~\ref{Section:TheMainTheorem} contains the proof the main theorem, Theorem~2.

This paper evolved from a manuscript prepared by Jonathan Brucks within his master's degree research, under the supervision of both authors, during the academic year 2010--2011. 

This collaboration was partially supported by NSF Grant DMS-0819590. The first author received support from the Faculty of Sciences Research Fund at Universidad de los Andes.

\section{Structures and Logics} \label{Section:StructuresAndLogics}

\subsection{Continuous metric structures} \label{subsection:Structures}

Although, for simplicity, we will focus on $[0,1]$-valued continuous metric structures, our results regarding continuous logic may be easily extended to bounded $\mathbb{R}$-valued structures, or even unbounded structures, if we decompose them into bounded ones (see Section~\ref{S: beyond [0,1]}.), thus we prefer to give the more general definition of continuous metric structure.

\begin{definition} \label{Definition:MetricStructure}
A \emph{continuous metric structure}\index{structure|ii} (or simply a \emph{structure}) $M$ consists of the following items:
\begin{enumerate}
\item A family $(M_{i},d_{i})_{i \in I}$ of metric spaces.

\item A collection of functions of the form
\[
	F: M_{i_{1}} \times \dots \times M_{i_{n}} \to M_{i_{0}},
\]
called the \emph{operations}\index{structure!operation|ii} of the structure, each of which is uniformly continuous on every bounded subset of its domain.

\item A collection of real-valued functions of the form
\[
	R: M_{i_{1}} \times \dots \times M_{i_{n}} \to \mathbb{R},
\]
called the \emph{predicates}\index{structure!predicate|ii} of the structure, each of which is uniformly continuous on every bounded subset of its domain.
\end{enumerate}
\end{definition}

The \emph{constants}\index{structure!constant|ii} of a structure are the $0$-ary operations of the structure.

The metric spaces $M_{i}$ are called the \emph{sorts}\index{structure!sort|ii} of $M$, and we say that $M$ is \emph{based on} $(M_{i})_{i \in I}$.  If $M$ is based on $(M_{i})_{i \in I}$, we will say that a structure is \emph{discrete}\index{structure!discrete|ii} if for each $i\in I$  the distinguished metric on the sort $M_i$ is the discrete metric, and all the predicates of $M$ take values in $\{0,1\}$. Note that if $M$ is a discrete structure, the uniform continuity requirement for the operations and predicates of $M$ is superfluous.  A structure is \emph{complete}\index{structure!complete|ii} if all of its sorts are complete metric spaces.

Let $M$ be a structure based on $(M_{i})_{i \in I}$.  If $(F_{j})_{j \in J}$ is a list of the operations of $M$, and $(R_{k})_{k \in K}$ is a list of the predicates of $M$, we may write
\[
	M = (M_{i},F_{j},R_{k})_{i \in I, j \in J, k \in K}.
\]

If $a_{1},\dots,a_{n} \in M_{i}$, we denote by $\bar{a}$ the list of elements $a_{1},\dots,a_{n}$ and write simply $\bar{a} \in M_{i}$.   When the context allows it, we also denote by $\bar{a}$ denote the tuple $(a_{1},\dots,a_{n})$ by $\bar{a}$.  If it becomes necessary to refer to the length of a list or tuple of elements $\bar{a}$, we denote it by $\ell(\bar{a})$.

Examples of non-discrete structures include normed spaces, Banach algebras, Banach lattices, and operator spaces.  For more examples, see \cite[Examples 2.2]{Henson-Iovino:2002} and \cite[Examples 2.1]{Ben-Yaacov-Berenstein-Henson-Usvyatsov:2008}.

If a continuous metric structure $M$ is based on a single metric space $(M,d)$, we say that $M$ is a \emph{one-sorted structure}\index{structure!one-sorted|ii}.  In this case we call $(M,d)$ the \emph{universe}\index{structure!universe|ii} of $M$. For simplicity, we shall restrict our attention to one-sorted metric structures.   

Note that, informally, we use the same letter to denote a structure and its universe.  We follow this convention throughout the paper.

Let $M$ be a continuous metric structure, and let $N$ be a metric space that extends $M$ and contains $M$ as a dense subset. By the uniform continuity condition in Definition~\ref{Definition:MetricStructure}, each operation or predicate of $M$ has a unique extension to an operation or predicate of $N$.  The \emph{completion}\index{structure!completion|ii} of a continuous metric structure $(M,F_{i},R_{j})_{i \in I, j \in J}$ is the structure $(\overline{M},\bar{F}_{i},\bar{R}_{j})_{i \in I, j \in J}$, where $\overline{M}$ is the metric completion of $M$ and $\bar{F}_{i},\bar{R}_{j}$ are the unique extensions of $F_{i},R_{j}$ from the appropriate powers of $M$ to the corresponding powers of $\overline{M}$.

We call a one-sorted continuous metric structure \emph{bounded}\index{structure!bounded|ii} if its universe is bounded. 

\subsection{Signatures} \label{subection:Signatures}

In order to treat metric structures $M$ model-theoretically, it is convenient to have a formal way of indexing the operations and predicates of $M$, and specifying moduli of uniform continuity for them; this is a \emph{signature} for $M$.

\begin{definition} \label{Definition:Signature}
Let $M$ be a bounded continuous metric structure with metric $d$. A \emph{signature} for $M$ is a pair $\mathbf{S}=(S,\mathcal{U})$, where:
\begin{enumerate}
\item
$S$ is a first-order vocabulary consisting of the following items: for each operation
 $F: M^{n} \to M$ of $M$ (respectively, predicate $R: M^{n} \to \mathbb{R}$ of $M$), a pair of the form  $(f,n)$ (respectively, $(P,n)$), where $f$ and $P$ are syntactic symbols called $n$-ary \emph{operation
symbol} and \emph{$n$-ary predicate symbol}, respectively. If $n=0$, $f$ is called a \emph{constant symbol}. In this context, $F$ is denoted $f^M$ and called the \emph{interpretation of $f$} and $R$  is
denoted $P^M$ and similarly called the \emph{interpretation of $P$}.
\item
$\mathcal{U}$ is a family of \emph{uniform continuity moduli} for the symbols in $S$, that is: for
each $n$-ary function symbol $f$ or predicate symbol $P$ of $S$ an associated function
$\delta:\mathbb{Q}\cap(0,1)\to \mathbb{Q}\cap(0,1)$ such that if $\bar a=a_1,\dots,a_n\in M$ and $\bar b=b_1,\dots,b_n\in M$,
\[
\sup_{1\le i\le n}d(a_i,b_i)< \delta(\epsilon)
\quad\Rightarrow\quad
d(f^M(\bar a),f^M(\bar b))\le\epsilon
\]
in the first case and
\[
\sup_{1\le i\le n}d(a_i,b_i)< \delta(\epsilon)
\quad\Rightarrow\quad
|P^M(\bar a)-P^M(\bar b)|\le\epsilon
\]
in the second.
\end{enumerate}
If $\mathbf{S}=(S,\mathcal{U})$ is a signature for $M$, we say  that $M$ is an \emph{$S$-structure} or an \emph{$\mathbf{S}$-structure}, depending on whether the uniform continuity moduli of $\mathcal{U}$ need to be made explicit for the context.

\end{definition}

If $S,S'$ are vocabularies, we write $S \subseteq S'$ if every operation and predicate symbol in $S$ is in $S'$ with the same arity.  In this case we say that $S'$ is an \emph{extension}\index{vocabulary!extension|ii} of $S$.

Let $S,S'$ be vocabularies with $S \subseteq S'$ and suppose that $N$ is an $S'$-structure.  The \emph{reduct}\index{structure!reduct|ii} of $N$ to $S$, denoted $N \upharpoonright S$, is the $S$-structure that results by removing from $N$ the operations and predicates that are indexed by $S'$ but not $S$.  We say that a continuous metric structure $N$ is an \emph{expansion}\index{structure!expansion|ii} of a structure $M$ if $M$ is a reduct of $N$.

Let $S$ be a vocabulary and suppose that $M$ and $N$ are $S$-structures.  We say that $M$ is a \emph{substructure}\index{structure!substructure|ii} of $N$ (or that $N$ is an \emph{extension}\index{structure!extension|ii} of $M$) if the following conditions hold:
\begin{itemize}
\item The universe of $N$ contains the universe of $M$, and the distinguished metric of $N$ extends the distinguished metric of $M$.

\item For every operation symbol $f$ of $S$, the operation $f^{N}$ extends $f^{M}$.

\item For every predicate symbol $P$ of $S$, the predicate $P^{N}$ extends $P^{M}$.
\end{itemize}

Let $S$ be a vocabulary and let $M,N$ be $S$-structures.  A \emph{metric isomorphism} between $M$ and $N$ is a surjective isometry $T: M \to N$ which commutes with the interpretation of the operation and predicate symbols of $S$.  We say that $M$ and $N$ are \emph{metrically isomorphic} (or simply \emph{isomorphic}\index{structure!isomorphic|ii}), and write $M \simeq N$, if there exists a metric isomorphism between $M$ and $N$.

Clearly, if $\mathbf{S}$ is a signature and $M$ is an $\mathbf{S}$-structure, then every structure that is metrically isomorphic to $M$ is also an $\mathbf{S}$-structure.

A \emph{renaming} is a bijection between vocabularies that sends function symbols to function symbols, predicate symbols to predicate symbols, and preserves arities.  If $M$ is an $S$-structure and  $\rho:S\to S'$ is a renaming, we denote by $M^\rho$ the $S'$-structure that results from converting $M$ into an $S'$-structure through $\rho$.

\subsection{Logics} \label{subection:Logics}

The formal definition of model-theoretic logic was introduced by P.~Lindstr\"{o}m in his famous paper \cite{Lindstrom:1969}. Lindstr\"{o}m's original definition of logic was intended for classical structures, i.e., discrete structures. Here we will use it for the more general context of continuous metric structure given in Definition~\ref{Definition:MetricStructure}. In general, throughout the paper, the word ``structure'' will stand for ``continuous metric structure''.

\begin{definition} \label{Definition:LogicalSystem}
A \emph{logic} $\mathcal{L}$ is a pair $(\Sent_{\mathcal{L}}, \models_{\mathcal{L}})$, where $\Sent_{\mathcal{L}}$ is a function that assigns to every vocabulary $S$ a set $\Sent_{\mathcal{L}}(S)$ called the set  of \emph{$S$-sentences of $\mathcal{L}$}\index{logic!sentence|ii} and $\models_{\mathcal{L}}$ is a binary relation between structures and sentences, such that the following conditions hold:
\begin{enumerate}
\item If $S \subseteq S'$, then $\Sent_{\mathcal{L}}(S) \subseteq \Sent_{\mathcal{L}}(S')$.

\item If $M \models_{\mathcal{L}} \varphi$ (i.e., if $M$ and $\varphi$ are related under $\models_{\mathcal{L}}$), then there is a vocabulary $S$ such that $M$ is an $S$-structure and $\varphi$ an $S$-sentence.

\item \emph{Isomorphism Property}\index{logic!Isomorphism Property|ii}.  If $M \models_{\mathcal{L}} \varphi$ and $M \simeq N$, then $N \models_{\mathcal{L}} \varphi$.

\item \emph{Reduct Property}\index{logic!Reduct Property|ii}.  Let $S \subseteq S'$ and suppose $\varphi$ is an $S$-sentence and $M$ an $S'$-structure.  Then
\[
	M \models_{\mathcal{L}} \varphi \qquad \text{if and only if} \qquad (M \upharpoonright S) \models_{\mathcal{L}} \varphi.
\]

\item \emph{Renaming Property}\index{logic!Renaming Property|ii}.  If $\rho: S \to S'$ is a renaming,  then for each $S$-sentence $\varphi$ there is an $S'$-sentence $\varphi^{\rho}$ such that $M \models_{\mathcal{L}} \varphi$ if and only if $M^{\rho} \models_{\mathcal{L}} \varphi^{\rho}$.  (Recall that $M^{\rho}$ denotes the structure that results from converting $M$ into an $S'$-structure through $\rho$.)
\end{enumerate}
If $M \models_{\mathcal{L}} \varphi$, we say that $M$ \emph{satisfies} $\varphi$, or that $M$ is a \emph{model} of $\varphi$.

\end{definition}

The study of abstract logics is known as abstract model theory.  For a survey, the reader is referred to~\cite{Barwise-Feferman:1985}.

A logic $\mathcal{L}$ is said to be \emph{closed under conjunctions}\index{conjunctions (closed under)|ii} if given any two $\mathcal{L}$-sentences $\varphi,\psi$ there exists an $\mathcal{L}$-sentence $\varphi \wedge \psi$ such that for every structure $M$
\[
	M \models_{\mathcal{L}} \varphi \wedge \psi \qquad \text{if and only if} \qquad M \models_{\mathcal{L}} \varphi \quad \text{and} \quad M \models_{\mathcal{L}} \psi.
\]
Similarly, $\mathcal{L}$ is said to be \emph{closed under disjunctions}\index{disjunctions (closed under)|ii} if given two $\mathcal{L}$-sentences $\varphi,\psi$ there exists an $\mathcal{L}$-sentence $\varphi \vee \psi$ such that for every structure $M$
\[
	M \models_{\mathcal{L}} \varphi \vee \psi \qquad \text{if and only if} \qquad M \models_{\mathcal{L}} \varphi \quad \text{or} \quad M \models_{\mathcal{L}} \psi.
\]
A logic $\mathcal{L}$ is said to be \emph{closed under negations}\index{negations (closed under)|ii} if given an $\mathcal{L}$-sentence $\varphi$ there exists an $\mathcal{L}$-sentence $\neg \varphi$ such that for every structure $M$
\[
	M \models_{\mathcal{L}} \neg \varphi \qquad \text{if and only if} \qquad M \nmodels_{\mathcal{L}} \varphi.
\]

Abstract logics without negation have been studied in \cite{Iovino:2001, Garcia-Matos:2004, Garcia-Matos-Vaananen:2005}.  

\begin{convention} \label{Convention:AssumedAndInfinitaryConnectives}
We will assume that all logics are closed under finite conjunctions and disjunctions, but not necessarily under negations.  We will also assume that every logic $\mathcal{L}$ mentioned is nontrivial in the following sense: for every structure $M$, there is a sentence $\varphi$ such that $M \nmodels_{\mathcal{L}} \varphi$. 
\end{convention}

\begin{definition} \label{Definitions:TheoryModelConsistent}
Let $S$ be a vocabulary.
\begin{enumerate}
\item An \emph{$S$-theory} (or simply a \emph{theory}\index{theory|ii} if the vocabulary is given by the context) is a set of $S$-sentences.

\item Let $T$ be an $S$-theory.  If $M$ is an $S$-structure such that $M \models_{\mathcal{L}} \varphi$ for each $\phi\in T$, we say that $M$ is a \emph{model}\index{theory!model|ii} of $T$ and write $M \models_{\mathcal{L}} T$.

\item A theory $T$ is \emph{consistent}\index{theory!consistent|ii} if it has a model.
\end{enumerate}
\end{definition}

If $\varphi \in \Sent_{\mathcal{L}}(S)$ for some vocabulary $S$, but there is no need to refer to the specific vocabulary, we may refer to $\varphi$ an \emph{$\mathcal{L}$-sentence}. Similarly, when $T$ is an $S$-theory for some vocabulary $S$ and there is no need to refer to $S$, we may refer to $T$ as an \emph{$\mathcal{L}$-theory}. 

If $S$ is a vocabulary, $\bar{x} = x_{1},\dots,x_{n}$ is a finite list of constant symbols not in $S$, and $\varphi$ is an $(S \cup \{\bar{x}\})$-sentence, we emphasize this by writing $\varphi$ as $\varphi(\bar{x})$.  In this case we may say that $\varphi(\bar{x})$ is an \emph{$S$-formula}.  If $M$ is an $S$-structure and $\bar{a} = a_{1},\dots,a_{n}$ is a list of elements of $M$, we write
\[
	(M,a_{1},\dots,a_{n}) \models_{\mathcal{L}} \varphi(x_{1},\dots,x_{n})
\]
or
\[
	M \models_{\mathcal{L}} \varphi[\bar{a}]
\]
if the $S \cup \{\bar{x}\}$ expansion of $M$ that results from interpreting $x_{i}$ by $ a_{i}$ (for $i = 1,\dots,n$) satisfies $\varphi(\bar{x})$.

\begin{definition} \label{Definition:LogicalEquivalence}
Let $M,N$ be $S$-structures.  We say that $M$ and $N$ are \emph{equivalent in $\mathcal{L}$}\index{structure!equivalent|ii}, and write $M \equiv_{\mathcal{L}} N$, if for every $S$-sentence $\varphi$ we have $M \models_{\mathcal{L}} \varphi$ if and only if $N \models_{\mathcal{L}} \varphi$.
\end{definition}

If $M$ is a structure and $A$ is a subset of the universe of $M$, we denote by $(M,a)_{a\in A}$ the expansion of $M$ that results by adding a constant symbol for each element of $A$. The structure  $(M,a)_{a\in A}$ is said to be an \emph{expansion of $M$ by constants}.

\begin{definition} \label{Definition:ElementarySubstructure}
Let $\mathcal{L}$ be a logic and let $M,N$ be $S$-structures with $M$ a substructure of $N$.  We say that $M$ is an \emph{elementary substructure}\index{structure!elementary substructure|ii} of $N$ (with respect to $\mathcal{L}$), and write $M \preceq_{\mathcal{L}} N$, if  $(M,a)_{a\in A}\equiv (N,a)_{a\in A}$.
\end{definition}

Recall that a signature is a pair $\mathbf{S}=(S,\mathcal{U})$, where $S$ is a first-order vocabulary and $\mathcal{U}$ is a family of uniform continuity moduli for the symbols of $S$ (see Definition~\ref{Definition:Signature}).

\begin{definition} \label{Definition:CompactnessProperty}
Let $\mathcal{L}$ be a logic.\hfill
\begin{enumerate}
\item
 \label{Definition:LambdaCompactLogic}
If $\lambda$ is an infinite cardinal, $\mathcal{L}$ is \emph{$\lambda$-compact} if whenever $\mathbf{S}$ is a signature and $T$ is an $S$-theory of cardinality~$\le\lambda$ such that every finite subset of $T$ is satisfied by an $\mathbf{S}$-structure, the theory $T$ is satisfied by an $\mathbf{S}$-structure.
\item
$\mathcal{L}$ is \emph{compact} if it is $\lambda$-compact for every infinite $\lambda$.
\end{enumerate}
\end{definition}

%
%
%
%
%
%

The following concept will be needed for the statement of the Main Theorem (Theorem~\ref{Theorem:TheMainTheorem}).

\begin{definition} 
 \label{D:FiniteOccurrenceProperty}
We say that a logic $\mathcal{L}$ has the \emph{finite occurrence property}\index{logic!finite occurrence property|ii} if for every vocabulary $S$ and every $S$-sentence $\varphi$ there is a finite vocabulary $S_0\subseteq S$ such that $\phi$ is an $S_0$-sentence.
\end{definition}

\subsection{$[0,1]$-valued logics} \label{subection:RealValuedLogics}
Hereafter, for simplicity,  we focus on $[0,1]$-valued structures, i.e., continuous metric structures where the distinguished metric and all the predicates take values on the closed unit interval $[0,1]$. More general structures are discussed in Section~\ref{S: beyond [0,1]}.

We now refine Lindstr\"om's definition of logic (Definition~\ref{Definition:LogicalSystem}):

\begin{definition}
\label{Definition:RealValuedLogic}
A  \emph{$[0,1]$-valued logic}  is a pair $(\Sent_{\mathcal{L}},\mathcal{V})$, where $\Sent_{\mathcal{L}}$ is a function that assigns to every vocabulary $S$ a set $\Sent_{\mathcal{L}}(S)$ called the set  of \emph{$S$-sentences of $\mathcal{L}$}\index{logic!sentence|ii} and $\mathcal{V}$ is a functional relation such that the following conditions hold: 
\begin{enumerate}

\item 
If $S \subseteq S'$, then $\Sent_{\mathcal{L}}(S) \subseteq \Sent_{\mathcal{L}}(S')$.
\item
The relation $\mathcal{V}$ assigns to every pair $(\phi, M)$, where $\phi$ is an $S$-sentence of $\mathcal L$ and  $M$ is an $S$-structure, a real number $\phi^M\in[0,1]$ called the \emph{truth value} of $\phi$ in $M$.



\item \emph{Isomorphism Property for $[0,1]$-valued logics}. If $M, N$ are metrically isomorphic structures of $\mathcal L$ and $\phi$ is an $S$-sentence of $\mathcal{L}$, then $\phi^M=\phi^N$.

\item \emph{Reduct Property  for $[0,1]$-valued logics.}  If $S \subseteq S'$,  $\varphi$ is an $S$-sentence of $\mathcal L$, and $M$ an $S'$-structure of $\mathcal{L}$, then
$\varphi^M =\phi^{M \upharpoonright S}$.

\item \emph{Renaming Property  for $[0,1]$-valued logics.} If $\rho: S \to S'$ is a renaming, then for each $S$-sentence $\varphi$ of $\mathcal{L}$ there is an $S'$-sentence $\varphi^{\rho}$ such that $\phi^M=(\phi^\rho)^{M^{\rho}}$  for every $S$-structure $M$. 
\end{enumerate}
\end{definition}

\begin{definition}
If $\mathcal L$ is a $[0,1]$-valued logic, $\phi$ is an $S$-sentence of $\mathcal L$ and  $M$ is an $S$-structure such that $\varphi^M=1$, we say that $M$ \emph{satisfies} $\varphi$, or that $M$ is a \emph{model} of $\varphi$, and write $M \models_{\mathcal{L}} \varphi$.
\end{definition}

Note that if $\mathcal L$ is  $[0,1]$-valued logic, then $(\Sent_{\mathcal{L}}, \models_{\mathcal{L}})$ is a logic in the sense of Definition~\ref{Definition:LogicalSystem}. Therefore we may apply to $\mathcal{L}$ all the concepts and properties defined so far for plain logics.

\begin{definition} \label{Definition:LukasiewiczImplication}
The  \emph{{\L}ukasiewicz implication}  is the function $\lukimp$ from $[0,1]^2$ into $[0,1]$ defined by
\[
x\lukimp y= \min\{1-x+y,1\}.
\]
\end{definition}

Note that $x\lukimp y=1$ is and only if $x\le y$.

\begin{definition}
\label{D:basic connectives}
We will say that a $[0,1]$-valued logic $\mathcal{L}$ is \emph{closed under the basic connectives} if the following conditions hold for every vocabulary $S$:
\begin{enumerate}
\item 
If $\phi,\psi \in \Sent_{\mathcal{L}}(S)$, then there exists a sentence $\phi\lukimp\psi \in\Sent_{\mathcal{L}}(S)$ such that $(\phi\lukimp\psi)^M=(\phi)^M\lukimp(\psi)^M$ for every $S$-structure $M$.
\item 
For each rational $r\in[0,1]$, the set $\Sent_{\mathcal{L}}(S)$ contains a sentence with constant truth value~$r$. These sentences are called the \emph{constants} of $\mathcal L$.
\end{enumerate}
\end{definition}

%

\begin{notation}
If $\mathcal{L}$ is a $[0,1]$-valued logic that is closed under the basic connectives $\phi$ is a sentence of $\mathcal{L}$, and $r$ is a constant of $\mathcal L$, we will write $\phi\le r$ and $\phi\ge r$, as abbreviations, respectively, of $\phi\lukimp r$ and $r\lukimp \phi$. 
\end{notation}

\begin{remark}
\label{R:reduction of [0,1]-valued to 2-valued}
Let $\mathcal{L}$ be a $[0,1]$-valued logic and let $S$ be a vocabulary. If $M$ is an $S$-structure of $\mathcal{L}$, $\phi$ is an $S$-sentence of $\mathcal{L}$, and $r$ is a constant of $\mathcal{L}$, then  $M\models_{\mathcal{L}}\phi\le r$ if and only if $\phi^M\le r$, and $M\models_{\mathcal{L}}\phi\ge r$ if and only if $\phi^M\ge r$; thus, the truth value $\phi^M$ is determined by either of the sets
\[
\{\,r\in\mathbb{Q}\cap[0,1] \mid M\models_{\mathcal{L}}\phi\le r\,\},\qquad
\{\,r\in\mathbb{Q}\cap[0,1] \mid M\models_{\mathcal{L}}\phi\ge r\,\}.
\]
\end{remark}

The following proposition will be invoked multiple times in the paper; the proof is left to the reader.

\begin{proposition} \label{Proposition:InequalitiesAndTheLukasiewiczImplication}
Let $\mathcal{L}$ be a $[0,1]$-valued logic that is closed under the basic connectives, let $\phi$ be an $S$-sentence of $\mathcal{L}$, and let $r,s$ be constants of $\mathcal L$. Then, for every $S$-structure $M$, one has
\begin{enumerate}
\renewcommand{\theenumi}{\textup{(\arabic{enumi})}}
\renewcommand{\labelenumi}{\theenumi}

\item $\phi^M \le  (\phi\ge r)^M$.

\item $((\phi\geq r) \geq s)^M = (\phi\geq (r+s-1))^M$.
\end{enumerate}
\end{proposition}

\begin{notation}
If $\mathcal{L}$ is a $[0,1]$-valued logic that is closed under the basic connectives and $\phi,\psi$ are sentences of $\mathcal{L}$, we write $\neg\phi$ and $\phi\lor\psi$, as abbreviations, respectively, of $\phi\lukimp 0$ and $(\phi\lukimp\psi)\lukimp\psi$, and $\phi\land\psi$ as an abbreviation of $\neg(\neg\phi\lor\neg\psi)$.
\end{notation}

Note that  for every $S$-structure $M$, one has
\begin{align*}
(\phi\le 0)^M =& 1-(\phi)^M,\\
(\phi\lor\psi)^M=&\max\{\phi^M,\psi^M\},\\
(\phi\land\psi)^M=&\min\{\phi^M,\psi^M\}.
\end{align*}
In particular, every $[0,1]$-valued logic that is closed under the basic connectives is closed under conjunctions and disjunctions. 

We will refer to any function from $[0,1]^n$ into $[0,1]$, where $n$ is a nonnegative integer, as am $n$-ary \emph{connective}.  The \L ukasiewicz implication and the Pavelka constants are continuous connectives, as are all the projections $(x_1,\dots x_n)\mapsto x_i$. The following proposition states that any other other continuous connective can be approximated by finite combinations of these.

\begin{proposition}\label{P:connectives approximation}
Let $\mathcal{C}$ be the class of connectives generated by the \L ukasiewicz implication, the Pavelka constants, and the projections through composition.  Then every continuous connective is a uniform limit of connectives in $\mathcal{C}$. 
\end{proposition}

\begin{proof}
Since $\mathcal{C}$ is closed under the connectives $\max\{x,y\}$ and $\min\{x,y\}$, by the Stone-Weierstrass Theorem for lattices \cite[pp.~241-242]{Gillman-Jerison:1976}, we only need to show that the connectives $rx$, where $r$ is a dyadic rational, can be approximated by connectives in $\mathcal{C}$. 

Notice that if $x\in[0,1]$,
\[
\frac{1}{2}x=\lim_n \left(\bigvee\limits_{i=1}^{n}\frac{i}{n}\wedge \lnot (x\lukimp \frac{i}{n})\right).
\]
Hence, since the truncated sum $\oplus:[0,1]^2\to[0,1]$ is in $\mathcal{C}$ (as $x\oplus y=\neg x\lukimp y$), so are all the connectives $(\frac{1}{2}x+\dots+\frac{1}{2^n})x$, for any positive integer $n$.
\end{proof}

The following concept will be invoked in the statement of the Main Theorem (Theorem~\ref{Theorem:TheMainTheorem}).

\begin{definition}\label{D:closed under classical quantifiers}
Let $\mathcal{L}$  be a $[0,1]$-valued logic. We say that $\mathcal{L}$ is  \emph{closed under existential quantifiers} if given any any $S$-formula $\varphi(x)$ there exists an $S$-formula $\exists x\phi$ such that for every $S$-structure $M$ one has $(\exists x\phi)^M=\sup_{a\in M}(\phi[a]^M)$. Similarly, we say that $\mathcal{L}$ is  \emph{closed under universal quantifiers} if given any any $S$-formula $\varphi(x)$ there exists an $S$-formula $\forall x\phi$ such that for every $S$-structure $M$ one has $(\forall x\phi)^M=\inf_{a\in M}(\phi[a]^M)$.
\end{definition}

\subsection{Basic continuous logic, continuous logic, and {\L}ukasiewicz-Pavelka logic}\index{continuous logic|ii} \label{subection:ContinuousLogic}

In this subsection we define two particular $[0,1]$-valued logics, namely, the continuous logic framework of \cite{Ben-Yaacov-Usvyatsov:2010} and the {\L}ukasiewicz-Pavelka logic (see, for example, Section 5.4 of~\cite{Hajek:1998}). Both are logics for continuous metric structures. Traditionally, both logics have focused on particular classes of structures: the emphasis in continuous logic is in complete structures with arbitrary uniform continuity moduli, while in {\L}ukasiewicz-Pavelka logic the focus has been on structures whose operations and predicates are 1-Lipschitz. However, as model-theoretic logics, both can be seen as restrictions to specific classes of structures of a more general framework that we introduce here and call \emph{basic continuous logic}.

The structures of basic continuous logic are all continuous metric structures. The class of sentences of this logic is defined as follows.

For a vocabulary $S$, the concept of $S$-term is defined as in first-order logic. If $t(x_1,\dots,x_n)$ is an $S$-term (where $x_1,\dots, x_n$ are the variables that occur in $t$), $M$ is an $S$ structure, and $a_1,\dots,a_n$ are elements of $M$, the interpretation $t^M[a_1,\dots,a_n]$ is defined as in first-order logic as well. The atomic formulas of $S$ are all the expressions of the form $d(t_1,t_2)$ or $R(t_1,\dots,t_n)$, where  $R$ is an $n$-ary predicate symbol of $S$. If $\varphi(x_1,\dots,x_n)$ is an atomic $S$-formula with variables $x_1,\dots,x_n$ and $a_1,\dots,a_n$ are elements of and $S$-structure $M$, the interpretation $\varphi^M[a_1,\dots,a_n]$ is defined naturally by letting
\[
  R(t_1,\dots,t_n)^M[a_1,\dots,a_n] = R^M(t_1^M[a_1,\dots,a_n],\dots,t_n^M[a_1,\dots,a_n])
\]
and
\[
 d(t_1,t_2)^M[a_1,\dots,a_n] = d^M(t_1^M[a_1,\dots,a_n],t_2^M[a_1,\dots,a_n]).
 \]
The $S$-formulas of basic continuous logic are the syntactic expressions that result from closing the atomic formulas of $S$ under the \L ukasiewicz implication, the Pavelka constants, and the existential quantifier; formally, the concept of $S$-formula and the interpretation $\phi^M$ of a formula $\phi$ in a given structure $M$ are defined inductively by the following rules:
\begin{itemize}
\item
All atomic formulas of $S$ are $S$-formulas.
\item
If $\phi(x_1,\dots,x_n)$ and $\psi(x_1,\dots,x_n)$ are $S$-formulas, then $\psi\lukimp\psi$ is a formula; if  $a_1,\dots,a_n\in M$,  the interpretation $(\psi\lukimp\psi)^M[a_1,\dots,a_n]$ is defined as $(\psi^M[a_1,\dots,a_n])\lukimp(\psi^M[a_1,\dots,a_n])$.
\item
If $\phi(x_1,\dots,x_n, x)$ is an $S$-formula, then $\exists x\phi$ is a formula; if  $a_1,\dots,a_n\in M$,  the interpretation $(\exists x\phi)^M[a_1,\dots,a_n]$ is defined as $\sup_{a\in M}(\phi[a]^M[a_1,\dots,a_n])$,
\item
For every rational $r\in[0,1]$ there is an $S$-formula, denoted also $r$, whose interpretation in any structure is the rational $r$.
\end{itemize}

We write $M\models_{\cl}\phi[a_1,\dots,a_n]$ if $\phi[a_1,\dots,a_n]^M=1$. A \emph{sentence} of basic continuous logic  is a formula without free variables, and the \emph{truth value} of a $S$-sentence $\phi$ in an $S$-structure $M$ is $\phi^M$.

Recall that in any $[0,1]$-valued logic that is closed under the basic connectives,  the expressions $\neg\phi$, $\phi\lor\psi$, $\phi\land\psi$, $\phi\le r$, and $\phi\ge r$ are written as abbreviations of $\phi\lukimp 0$, $(\phi\lukimp\psi)\lukimp\psi$,  $\neg(\neg\phi\lor\neg\psi)$, $\phi\lukimp r$, and $r\lukimp\phi$, respectively. In basic continuous logic we also regard $\forall x\phi$ as an abbreviation of $\neg\exists x\neg\phi$. 

Having defined basic continuous logic, let us now describe the $[0,1]$-valued logic that Ben Yaacov and Usvyatsov introduced in~\cite{Ben-Yaacov-Usvyatsov:2010} and called \emph{continuous first-order-logic}. We will refer to this framework simply as continuous logic. It is an instance of the more general concept of continuous logic studied by Chang and Keisler in~\cite{Chang-Keisler:1966}, and was proposed by Ben Yaacov and Usvyatsov as a reformulation of Henson's model theory of complete metric spaces\footnote{For a survey of Henson's logic as it regards structures based on normed spaces, see~\cite{Henson-Iovino:2002}. The more general framework devised by Henson for arbitrary continuous metric structures was never published.}

The class of structures of continuous logic is the class of complete metric structures. The concepts of formula and sentence for this logic are defined as for basic continuous logic, but instead of taking the closure under the \L ukasiewicz-Pavelka connectives one takes the closure under all continuous connectives (and the existential quantifier). Proposition~\ref{P:connectives approximation} gives us the following:
\begin{remark}
\label{R:formulas approximation}
For every $S$-formula $\phi(\bar x)$ of continuous logic and for every $\epsilon>0$ there exists a formula $\psi(\bar x)$ of basic continuous logic such that $|\phi^M[\bar a]-\psi^M[\bar a]|<\epsilon$ for every $S$-structure $M$ and every tuple $\bar a$ in $M$ with $\ell(\bar a)=\ell(\bar x)$.
\end{remark}
This observation allows us to transfer model-theoretic results from basic continuous logic to continuous logic, by simply restricting them to the realm of complete metric structures.

Finally, let us discuss \L ukasiewicz-Pavelka logic. Pavelka extended  \L ukasiewicz propositional logic by adding the rational constants, and proved a form of approximate completeness for  the resulting logic. See~\cite{Pavelka:1979I,Pavelka:1979II,Pavelka:1979III} (see also Section 5.4 of~\cite{Hajek:1998}.) This is known as Pavelka-style completeness. The extension of \L ukasiewicz logic with Pavelkas's constants is referred to in the literature as rational Pavelka logic, or Pavelka many-valued logic. Nov\'ak proved Pavelka-style completeness for predicate \L ukasiewicz-Pavelka logic, which he calls ``first-order fuzzy logic", first using ultrafilters~\cite{Novak:1989,Novak:1990}, and later using a Henkin-type construction~\cite{Novak:1995}. Another proof of  Pavelka-style completeness for predicate \L ukasiewicz-Pavelka logic was given by Hajek; see~\cite {Hajek:1997} and~\cite[Section 5.4]{Hajek:1998}.  Hajek, Paris, and Stepherson have proved that  \L ukasiewicz-Pavelka logic is a conservative extension of  \L ukasiewicz predicate logic~\cite{Hajek-Paris-Stepherson:2000}.

The formulas of \L ukasiewicz-Pavelka logic are like those of basic continuous logic, with the difference that  in place of the distinguished metric $d$ one uses the \emph{similarity} relation $x\approx y$. However, there is a precise correspondence between the two relations, namely, $d(x,y)$ is $1-(x\approx y)$ (in other words, the two relations are negations of each other); see Section~5.6 of~\cite{Hajek:1998}, especially Example~5.6.3-(1). Also, in  \L ukasiewicz-Pavelka logic, in place of the uniform continuity requirement given in Definition~\ref{Definition:Signature}, for each $n$-ary operation symbol $f$, one has the axiom
\label{P:similarity axioms}
\[
(x_1\approx y_1\land\dots\land x_n\approx y_n) \lukimp 
(f(x_1\dots,x_n)\ \approx f(y_1,\dots y_n)),
\]
and similarly, for each $n$-ary predicate symbol $R$, one has the axiom\footnote{Here we use $(\alpha\leftrightarrow_{L}\beta)$ as an abbreviation of $(\alpha \lukimp \beta)\land (\beta \lukimp \alpha)$.}
\[
(x_1\approx y_1\land\dots\land x_n\approx y_n) \lukimp 
(R(x_1\dots,x_n)\ \leftrightarrow_{L} R(y_1,\dots y_n)).
\]
See Definition~5.6.5 of~\cite{Hajek:1998}. Thus,  \L ukasiewicz-Pavelka logic is the restriction of basic continuous logic to the class of 1-Lipschitz structures, i.e., structures whose  operations and predicates are 1-Lipschitz.


Below, we state, without proof, some of the fundamental properties of basic continuous logic.
Versions of Theorems~\ref{Theorem:CompactnessTheoremForContinuousLogic}, \ref{Theorem:TarskiVaughtTestForContinuousLogic}, and \ref{Theorem:DownwardLowenheimSkolemTheoremForContinuousLogicComplete},  were first proved by Henson in the mid 1970's, for Banach spaces instead of general metric structures, and using Henson's logical formalism of positive bounded formulas instead of $[0,1]$-valued logics. (Henson's apparatus was one of the main motivations for the development of continuous logic). 

The most distinctive model-theoretic property of basic continuous logic is compactness:

\begin{theorem}[Compactness of Basic Continuous Logic] \label{Theorem:CompactnessTheoremForContinuousLogic}
Let $\mathbf{S}$ be a signature and let $T$ be an $S$-theory.  If every finite subtheory of $T$  is satisfied by an $\mathbf{S}$-structure, then $T$ is satisfied by an $\mathbf{S}$-structure; furthermore, this structure can be taken to be complete.
\end{theorem}

Theorem~\ref{Theorem:CompactnessTheoremForContinuousLogic} can be proved  by taking ultraproducts of $[0,1]$-valued structures. The argument is simple, but we omit the details, as they are not directly relevant to the rest of the paper.

One  one obtains compactness for continuous logic by restricting Theorem~\ref{Theorem:CompactnessTheoremForContinuousLogic} to complete structures and invoking Remark~\ref{R:formulas approximation}. Restricting the theorem to 1-Lipschitz structures yields compactness of  \L ukasiewicz-Pavelka logic, but it must be noted that in this case the 1-Lipschitz condition makes it unnecessary to fix uniform continuity moduli; in other words, for  \L ukasiewicz-Pavelka logic, the statement of Theorem~\ref{Theorem:CompactnessTheoremForContinuousLogic} holds with  the signature $\mathbf S$ replaced by  a vocabulary $S$.\footnote{Note that a modulus of uniform continuity with $\delta(\epsilon)=\epsilon$ does not guarantee 1-Lipschitz continuity.}

If $M$ is a structure and $A$ is a subset of the universe of $M$, we denote by $\langle A\rangle$ the closure of $A$ under the functions of  $M$, and by $M\upharpoonright \langle A\rangle$ the substructure of $M$ induced by $\langle A\rangle$.

\begin{theorem}[Tarski-Vaught Test for Basic Continuous Logic]
\label{Theorem:TarskiVaughtTestForContinuousLogic}
Let $M$ be an $S$-structure. If $A$ is a subset of $M$ and $S(A)$ is extension of $S$ that includes a constant symbol for each element of $A$, the following conditions are equivalent:
\begin{enumerate}
\renewcommand{\theenumi}{\textup{(\arabic{enumi})}}
\renewcommand{\labelenumi}{\theenumi}

\item $M\upharpoonright \langle A\rangle \preceq_{\cl} M$, and $\langle A \rangle$ is contained in the closure of $A$ in $M$.

\item For every $S(A)$-formula $\varphi(x)$, if $N\models_{\cl} \varphi[b]$ for some element $b$ of $N$, then for every rational $r \in(0,1)$ there is an element $a_{r}$ of $M$ such that $N \models_{\cl} \varphi[a_{r}] \geq r$.
\end{enumerate}
\end{theorem}

A direct consequence of Theorem~\ref{Theorem:TarskiVaughtTestForContinuousLogic} is the following property:

\begin{corollary}
 \label{C:StructuresAreElementarySubstructuresOfCompletions}
Let $M$ be a metric structure and let $N$ be a substructure of $M$.
\begin{enumerate}
\renewcommand{\theenumi}{\textup{(\arabic{enumi})}}
\renewcommand{\labelenumi}{\theenumi}
\item
If $N$ is dense in $M$, then $N\prec_{\cl}M$.
\item
If $N\prec_{\cl}M$, and $\overline{N}$ is the closure of $N$ in $M$, then $\overline{N}\prec_{\cl}M$
\end{enumerate}
\end{corollary}

It follows from the main result of \cite{Iovino:2001} that no $0,1]$-valued logic for continuous metric structures  that extends basic continuous logic properly satisfies  the compactness property and the Tarski-Vaught Test.  In other words, basic continuous logic is maximal with respect to these two properties. 

Now we turn our attention to the downward L\"{o}wenheim-Skolem-Tarski Theorem for basic continuous logic.  We state two versions of this theorem; the first version is for arbitrary continuous metric structures (Theorem~\ref{Theorem:DownwardLowenheimSkolemTheoremForContinuousLogic}) and the second one is for complete structures (Theorem~\ref{Theorem:DownwardLowenheimSkolemTheoremForContinuousLogicComplete}). In the continuous logic literature, where complete structures are emphasized, only the second version is usually stated; however, the standard argument used to prove the second version proceeds by first showing that the first version holds, and then taking completions and invoking Corollary~\ref{C:StructuresAreElementarySubstructuresOfCompletions}. We state the first version as a separate theorem, for it will be needed in the proof of Lemma~\ref{Lemma:ProjectionsAreContinuousOpenAndxive}). 

\begin{theorem}[L\"{o}wenheim-Skolem-Tarski Theorem for Continuous Metric Structures]\label{Theorem:DownwardLowenheimSkolemTheoremForContinuousLogic}
For every $S$-structure $M$ and every subset $A$ of $M$ there exists a substructure $N$ of $M$ such that 
\begin{enumerate}
\renewcommand{\theenumi}{\textup{(\arabic{enumi})}}
\renewcommand{\labelenumi}{\theenumi}
\item
$N\prec_{\cl} M$,
\item
$A\subseteq N$,
\item
$|N|\le |A|+|S|+\aleph_0$.
\end{enumerate}
\end{theorem}

The version for complete metric structures is analogous, but in this context the ``correct'' measure of size of a structure is its density, rather than its cardinality. The \emph{density character} (or simply \emph{density}) of a metric space $M$, denoted $\density(M)$, is the smallest cardinal of a dense subset of $M$.  

\begin{theorem}[L\"{o}wenheim-Skolem-Tarski Theorem for Complete Metric Structures]
\label{Theorem:DownwardLowenheimSkolemTheoremForContinuousLogicComplete}
For every complete $S$-structure $M$ and every subset $A$ of $M$ there exists a complete substructure $N$ of $M$ such that 
\begin{enumerate}
\renewcommand{\theenumi}{\textup{(\arabic{enumi})}}
\renewcommand{\labelenumi}{\theenumi}
\item
$N\prec_{\cl} M$,
\item
$A\subseteq N$,
\item
$\density(N)\le \density(A)+|S|+\aleph_0$.
\end{enumerate}
\end{theorem}

\subsection{Relativizations to discrete predicates}

It is often helpful to know that the fact that a predicate is discrete, in the sense that it only takes on values in $\{0,1\}$, can be expressed using formulas of basic continuous logic:

\begin{definition} \label{Definition:DiscretePredicate}
Let $M$ be an $S$-structure and let $P$ a predicate symbol of $S$.  We define $\Discrete \big( P(\bar{x}) \big)$ to be the $S$-formula
\[
	P(\bar{x}) \vee \neg P(\bar{x}),
\]
and call $P^{M}$ \emph{discrete}\index{structure!predicate!discrete} if $M \models \forall \bar{x} \, \Discrete \big( P(\bar{x}) \big)$.
\end{definition}

Definition~\ref{Definition:DiscretePredicate} will play an important role in the proof of the Main Theorem (Theorem~\ref{Theorem:TheMainTheorem}). 

Let $S$ be a vocabulary and let $P(x)$ be a monadic predicate not in $S$. If $M$ is a $(S\cup\{P\})$-structure such that $P^{M}$ is discrete, and a valid $S$-structure of $\mathcal{L}$ is obtained by restricting the universe of $M$ to $\{\,a\in M\mid M\models_{\mathcal{L}}P[a]\,\}$, we denote this structure by $M\upharpoonright \{x \mid P(x)\}$.
Note that if $M$ is complete, the continuity of $P$ ensures that the preceding structure, when defined,  is complete. If $\phi$ is an $S$-formula of continuous logic,  the \emph{relativization} of $\phi$ to $P$, denoted $\phi^{\{x\mid P(x)\}}$ or $\phi^P$, is the $(S\cup\{P\})$-formula  defined by the following recursive rule:

\begin{itemize}
\item
If $\phi$ is atomic, then $\phi^{\{x\mid P(x)\}}$ is $\phi$.
\item
If $\phi$ is of the form $C(\psi_1,\dots,\psi_n)$, where $C$ is a connective, then $\phi^{\{x\mid P(x)\}}$ is $C(\psi_1^{\{x\mid P(x)\}},\dots,\psi_n^{\{x\mid P(x)\}})$.
\item
If $\phi$ is of the form $\exists x \psi$, then $\phi^{\{x\mid P(x)\}}$ is $\exists y(\psi(y) \land \psi^{\{x\mid P(x)\}})$.
\item
If $\phi$ is of the form $\forall x \psi$, then $\phi^{\{x\mid P(x)\}}$ is $\forall y(\neg \psi(y) \lor \psi^{\{x\mid P(x)\}})$.
\end{itemize}
Then, if $M$ and $P$ are as above, we have
\[
(\phi^{\{x\mid P(x)\}})^M=
\phi^{M\upharpoonright \{x \mid P(x)\}}.
\]

In the preceding rule, instead of the monadic predicate $P(x)$ we may have a binary predicate symbol $R(x,y)$; this shows that the logics discussed in this section have the following property:

\begin{definition}\label{D:logic permits relativization}
We will say that a $[0,1]$-valued logic $\mathcal{L}$ permits \emph{relativization to definable families of predicates} if for every vocabulary $S$, every $S$-sentence $\phi$ and every predicate predicate symbol $R(x,y)$ not in $S$ there is an $(S\cup\{R\})$-formula $\psi(x)$, denoted $\phi^{\{y\mid R(x,y)\}}(x)$, such that the following holds: whenever $M$ is an $(S\cup\{R\})$-structure such that for every $a$ in $M$,
\begin{itemize}
\item
either $M\models R[a,b]$ or $M\models \neg R[a,b]$ for every $b$ in $M$, and
\item
$(M, a)\upharpoonright \{y \mid R(x,y)\}$
 is defined as a structure of $\mathcal{L}$, 
 \end{itemize}
 one has
 \[
(\phi^{\{y\mid R(x,y)\}})^M [a]= \phi^{M\upharpoonright \{y \mid R(x,y)\}}[a].
\]
In this case, we call the formula $\phi^{\{y\mid R(x,y)\}}(x)$ a \emph{relativization} of $\phi$ to ${\{\,y\mid R(x,y)\,\}}$.
\end{definition}

A logic $\mathcal L$ has the  L\"{o}wenheim-Skolem property for sentences if every sentence of $\mathcal L$  that has a model has a countable model. The first author has proved the following version of Lindstr\"om's First Theorem~\cite{Caicedo:201?}:

\begin{theorem}
Let $\mathcal{L}$ be an extension of continuous logic  that satisfies the following properties:
\begin{itemize}
\item
Closure under the \L ukasiewicz-Pavelka connectives and relativization to discrete predicates,
\item
Compactness,
\item
The L\"{o}wenheim-Skolem property for sentences.
\end{itemize}
Then $\mathcal{L}$ is equivalent to continuous logic.

A similar result holds for extensions of first order \L ukasiewicz-Pavelka logic, under the additional assumption that any consistent theory
has a complete model.
\end{theorem}

\subsection{Beyond the interval $[0,1]$}\label{S: beyond [0,1]}
For simplicity, we have focused our attention on continuous metric structures of diameter bounded by 1. The formalism of basic continuous logic can be adapted to cover bounded structures of arbitrary diameter.  However,  the resulting logic is not compact, but rather locally compact, in the following sense:

\begin{theorem}
\label{T:local compactness for unbounded structures}
Let $\mathbf{S}$ be a signature. For every bounded  $\mathbf{S}$-structure $M$ there is an $S$-sentence $\varphi$  and a rational $r\in(0,1)$ such that $M \models \phi$ and the following property holds. If $T$ is an $S$-theory such that every finite subset of $T \cup\{\varphi\ge r\}$ is satisfied by an $\mathbf{S}$-structure, then $T \cup\{\varphi\ge r\}$ is satisfied by an $\mathbf{S}$-structure.
\end{theorem}

The main results of the paper, namely, Theorems~\ref{Theorem:GeneralOmittingTypesTheorem} and~\ref{Theorem:TheMainTheorem} hold under this wider semantics, although the second result holds
locally only. One way to obtain this generalization is to use $0$ instead of 1 as designated truth value for $\models$, and the truncated subtraction on $[0,1]$ instead of the \L ukasiewicz implication.


\section{Logics and Topologies} \label{Section:TheLogicalTopology}

In this section we associate with every logic $\mathcal{L}$ a topology that we call the \emph{logical topology} of $\mathcal{L}$.  The idea of using the logical topology  to study properties of abstract logics is due to the first author \cite{Caicedo:1993,Caicedo:1995,Caicedo:1999}.  We will focus on logics whose logical topology is regular.

If  $\mathcal{L}$ is a logic and $\mathbf{S}$ is a signature, we will denote by $\Str_{\mathcal L}(\mathbf{S})$ the class of $\mathbf{S}$-structures of $\mathcal{L}$. For every $S$-theory $T$ of $\mathcal{L}$, define
\[
\Mod_{\mathbf{S}}(T) =\{\,M\in\Str_{\mathcal L}(\mathbf{S}) \mid M\models_{\mathcal{L}}T \,\}.
\]

The following is the main definition of this section. This concept will play a central role in our arguments.

\begin{definition} \label{Definition:LogicalTopology}
The  \emph{logical topology}\index{logical topology} on $\Str_{\mathcal L}(\mathbf{S})$, denoted $\tau_{\mathcal{L}}(\mathbf{S})$, is the topology on $\Str_{\mathcal L}(\mathbf{S})$ whose closed classes are those of  $\Mod_{\mathbf{S}}(T)$, where $T$ is a theory.
\end{definition}

The following proposition follows directly from the definitions.

\begin{proposition} \label{Proposition:EquivalenceOfLogicalAndTopologicalCompactness}
Let $\mathcal{L}$ be a logic.
\begin{enumerate}
\renewcommand{\theenumi}{\textup{(\arabic{enumi})}}
\renewcommand{\labelenumi}{\theenumi}

\item $\mathcal{L}$ is compact if and only if the space $(\,\Str_{\mathcal L}(\,\mathbf{S}),\tau_{\mathcal{L}}(\mathbf{S})\,)$ is compact, for every signature $\mathbf{S}$.

\item $\mathcal{L}$ is $\lambda$-compact if and only if the space $(\,\Str_{\mathcal L}(\mathbf{S}),\tau_{\mathcal{L}}(\mathbf{S})\,)$ is $\lambda$-compact, for every signature $\mathbf{S}$.
\end{enumerate}
\end{proposition}

Since we are assuming that all logics are closed under finite disjunctions (see Convention~\ref{Convention:AssumedAndInfinitaryConnectives}), the classes of the form
\[
\Mod_{\mathbf{S}}(\varphi) =\{\,M\in\Str_{\mathcal L}(\mathbf{S}) \mid M\models_{\mathcal L}\varphi \,\},
\]
where $\varphi\in\Sent_{\mathcal L}(S)$, are closed under finite unions. These classes form a base of closed classes for the logical topology on $\Str_{\mathcal L}(\mathbf{S})$.

\begin{convention}
Let $\mathcal L$ be a $[0,1]$-valued logic that is closed under the basic connectives and let $\mathbf S$ be a signature. If $\phi$ is a sentence of $\mathcal L$, and $r\in[0,1]$ is a constant of $\mathcal{L}$,  we write
\[
\Mod_{\mathbf{S}}(\phi<r)\qquad\text{and} \qquad \Mod_{\mathbf{S}}(\phi>r),
\]
respectively, as abbreviations for the classes
\[
\Str_{\mathcal L}(\mathbf{S})\setminus\Mod_{\mathbf{S}}(\phi\ge r)\
\qquad\text{and}\qquad 
\Str_{\mathcal L}(\mathbf{S})\setminus\Mod_{\mathbf{S}}(\phi\le r).
\]
\end{convention}

Note that the classes of the form $\Mod_{\mathbf{S}}(\phi<1)$ where $\phi\in\Sent_{\mathcal L}(S)$ for a base for $\tau_{\mathcal{L}}$. By replacing $\phi$ with $\neg \phi$, it follows that the classes  of the form $\Mod_{\mathbf{S}}(\phi<1)$ for a base for $\tau_{\mathcal{L}}$ as well.

\subsection{Comparing logics through their topologies}

\begin{definition} \label{Definition:ComparingLogicsAndEquivalentLogics}
Let $\mathcal{L},\mathcal{L}'$ be logics with the same class of structures.  
We will say that $\mathcal{L}'$ \emph{extends}\index{logic!rxtension|ii} $\mathcal{L}$, and write $\mathcal{L} \prec \mathcal{L}'$, if $\tau_{\mathcal{L}}(\mathbf{S}) \subseteq \tau_{\mathcal{L}'}(\mathbf{S})$ for every signature $\mathbf{S}$.  
We will say that $\mathcal{L}$ and $\mathcal L'$ are \emph{equivalent}\index{logic!equivalent|ii}, and write $\mathcal{L} \sim \mathcal{L}'$, if both $\mathcal{L} \prec \mathcal{L}'$ and $\mathcal{L}' \prec \mathcal{L}$ hold.
\end{definition}

Recall that a topological space $X$ is regular if and only it has a local base of closed neighborhoods, i.e., whenever $x\in X$ and $U$ is a neighborhood of $x$ there exists a neighborhood $W$ of $x$ such that $\overline{W}\subseteq U$. We shall focus our attention on logics whose logical topology is regular. The following observation is useful to compare logics.



If $(X,\tau)$ is topological space and $x,y \in X$, one says that $x$ and $y$ are \emph{$\tau$-indistinguishable}, denoted $x \overset{\tau}{\equiv} y$, if every $\tau$-neighborhood of $x$ contains $y$ and every $\tau$-neighborhood of $y$ contains $x$.
Note that if $(X,\tau)$ is a regular topological space and $x,y \in X$, then $x \overset{\tau}{\equiv} y$ if and only if every $\tau$-neighborhood of $x$ contains $y$.  Also, $\big\{\, z \in X \mid z \overset{\tau}{\equiv} x \,\big\}$ is the $\tau$-closure of $\{x\}$.  If $\tau, \tau'$ are topologies on $X$ and any two $\tau$-indistinguishable points of $X$ are also $\tau'$-indistinguishable, we will write $\overset{\tau}{\equiv} \To \overset{\tau'}{\equiv}$.

If $\mathcal L$ is a logic and $M, N$ is a structure of $\mathcal L$, then, clearly, $M \equiv_{\mathcal{L}} N$ if and only if $M$ and $N$ are indistinguishable in the logical topology of $\mathcal L$. If $\mathcal L$ and $\mathcal L'$ are logics with the same class of structures, we write $\equiv_{\mathcal{L}} \To \equiv_{\mathcal{L}'}$ if $\overset{\tau_{\mathcal{L}}(S)}{\equiv} \To \overset{\tau_{\mathcal{L}'}(S)}{\equiv}$ for every vocabulary $S$.

Suppose that $\tau$ and $\tau'$ are topologies on a set $X$ such that $\tau\subseteq\tau'$. Clearly, $\overset{\tau'}{\equiv} \To \overset{\tau}{\equiv}$. The following proposition shows that under sufficient compactness of $\tau'$, from $\overset{\tau}{\equiv} \To \overset{\tau'}{\equiv}$ one can obtain $\tau=\tau'$. 

Recall that the \emph{weight} of a topological space $(X,\tau)$ is the smallest possible cardinality of a base for $\tau$.

\begin{proposition} \label{Proposition:ApproximabilityIndistinguishabilityAndCompactnessImpliesReverseApproximability}
Let $\tau,\tau'$ be regular topologies on $X$ such that $\tau\subseteq\tau'$ and $\overset{\tau}{\equiv} \To \overset{\tau'}{\equiv}$. If $\tau$ has weight $\lambda$ and  and $\tau'$ is $\lambda$-compact, then $\tau'= \tau$.
\end{proposition}

\begin{proof}
Let $\mathcal B$ be a base for $\tau$ of cardinality $\lambda$, and fix $U\in \tau'$ in order to prove $U\in\tau$. Fix now $x\in U$ and $y\in U^c$. Since $x$ and $y$ are $\tau'$-topologically distinguishable and $\overset{\tau}{\equiv} \To \overset{\tau'}{\equiv}$, there exist disjoint $\tau$-open sets $V_{x,y}, W_{x,y}$ such that $x\in V_{x,y}$ and $y\in W_{x,y}$. Without loss of generality, we may assume $W_{x,y}\in\mathcal B$. Since $W_{x,y}\in\tau\subseteq\tau'$, allowing $y$ to range over all elements of $U^c$ and using the $\lambda$-compactness of $\tau'$, we obtain $V_x\in\tau$ such that $x\in V_x$ and $V_x\subseteq U$. Since $x$ is arbitrary in $U$, this shows that $U$ is $\tau$-open.
\end{proof}

\subsection{$[0,1]$-valued logics and regularity} \label{S:real-valued logics and regularity}
We now link the concept of $[0,1]$-valued logic of Definition~\ref{Definition:RealValuedLogic} with the logical topology introduced in this section.

\begin{proposition}
\label{P:real-valued logic continuity regularity}
Let $\mathcal L$ be a $[0,1]$-valued logic that is closed under the basic connectives and let $\mathbf S$ be a signature.
\begin{enumerate}
\renewcommand{\theenumi}{\textup{(\arabic{enumi})}}
\renewcommand{\labelenumi}{\theenumi}
\item
\label{I:real-valued logic continuity}
Every $S$-sentence is a continuous function from $(\Str_{\mathcal L}(\mathbf{S}),\tau_\mathcal{L}(\mathbf{S}))$ into $[0,1]$.
\item
\label{I:real-valued logic regularity}
$(\Str_{\mathcal L}(\mathbf{S}),\tau_\mathcal{L}(\mathbf{S}))$ is a regular topological space.
\end{enumerate}
\end{proposition}

\begin{proof}
Both conclusions follow readily from the observation that if $\phi$ is a sentence and $r,s$ are rationals with $0\le r<s\le 1$, then
$\phi^{-1}[r,s]= \Mod_{\mathbf{S}}(\phi\ge r\land\phi\le s)$.
\end{proof}

\section{A General Omitting Types Theorem} 
\label{Section:A General Omitting Types Theorem}

The main result of this section is Theorem~\ref{Theorem:GeneralOmittingTypesTheorem}.

\subsection{A general version of the Baire Category Theorem} \label{Section:AnUncountableVersionOfTheBaireCategoryTheorem}

Let $X$ be a topological space and let $x \in X$ be given.  Recall that a subclass $G$ of $X$ is a \emph{neighborhood of $x$}\index{neighborhood|ii} if there is an open subclass $O$ of $G$ containing $x$.  If $\lambda$ is any infinite cardinal, we will say that a subclass $G$ of $X$ is a \emph{$\lambda$-neighborhood of $x$}\index{lambda-neighborhood@$\lambda$-neighborhood|ii} if $G$ is an intersection of less than $\lambda$ neighborhoods of $x$.  We will say that $G$ is a \emph{$\lambda$-neighborhood} if $G$ is a $\lambda$-neighborhood of some point in $X$.  This definition of $\lambda$-neighborhood yields naturally definitions concepts of $\lambda$-open, $\lambda$-interior, $\lambda$-dense, etc.

\begin{definition}  \label{Definition:LambdaBaireProperty}
Let $X$ be a topological space and $\lambda$ an infinite cardinal.  We will say that $X$ has the \emph{$\lambda$-Baire Property}\index{lambda-Baire Property@$\lambda$-Baire Property|ii} if whenever $\{D_{i}\}_{i < \lambda}$ is a collection of subclasses of $X$ that are $\lambda$-open and $\lambda$-dense, the intersection $\bigcap_{i < \lambda} D_{i}$ is $\lambda$-dense.
\end{definition}

Note that the $\omega$-Baire Property is simply the classical Baire Property.  Also, a topological space $X$ has the $\lambda$-Baire Property if and only if a union of $\leq \lambda$ many $\lambda$-closed classes with empty $\lambda$-interior has empty $\lambda$-interior.  This is immediate by taking complements in Definition~\ref{Definition:LambdaBaireProperty}.

A topological space $X$ is regular if and only it has a local base of closed neighborhoods; this means that whenever $x\in X$ and $U$ is a neighborhood of $x$ there exists a neighborhood $W$ of $x$ such that $\overline{W}\subseteq U$. Note that this form of regularity is inherited by passing to $\lambda$-neighborhoods, i.e., if  $X$ is regular and $U$ is a $\lambda$-neighborhood of $x$, then  there exists a $\lambda$-neighborhood $W$ of $x$ such that $\overline{W}\subseteq U$.

\begin{proposition} \label{Proposition:UncountableVersionOfBaireCategoryTheorem}
Every locally compact regular topological space has the $\lambda$-Baire property for every infinite cardinal $\lambda$.
\end{proposition}

\begin{proof}
Suppose $\{D_{i}\}_{i < \lambda}$ are $\lambda$-open and $\lambda$-dense in $X$, and let $O$ be a nonempty $\lambda$-open subclass of $X$.  We must show that $O \cap \bigcap_{i < \lambda} D_{i}$ is nonempty. Below we define, recursively, a decreasing sequence $(U_{i})_{i < \lambda}$ of $\lambda$-neighborhoods such that 
\begin{enumerate}
\item $\overline{U}_0$ is compact and $\overline{U}_0\subseteq O$,
\item $\overline{U}_{i} \subseteq D_{i}$,  for $i<\lambda$,
\item $\overline{U}_{j} \subseteq U_i$ for $j<i<\lambda$.
\end{enumerate}
The construction is as follows. For $i=0$,  the set $D_0\cap O$ is nonempty by hypothesis and $\lambda$-open by construction, so, by regularity and local compactness, there exists a $\lambda$-neighborhood $U_0$ such that  $\overline{U}_0\subseteq D_0\cap O$ and $\overline{U}_0$ is compact. 

Assume that $i<\lambda$ is positive and $U_j$ has been defined for $j<i$ with the properties indicated above. Then $\bigcap_{j<i} U_j\ne \emptyset$: if $i$ is successor, say, $i=j_0+1$, this is because $\bigcap_{j<i} U_j=U_{j_0}$, and if $i$ is limit, $\bigcap_{j<i} U_j= \bigcap_{j<i} \overline{U}_j$, which is nonempty by the compactness of $\overline{U}_0$. Since $D_i$ is  $\lambda$-dense, $D_i\cap  \bigcap_{j<i} U_j$ is nonempty, and by regularity  we can find a $\lambda$-neighborhood $U_i$ such that $\overline{U}_i\subseteq D_i\bigcap_{j<i} U_j$. This concludes the recursive definition. Finally, by compactness, we obtain $\emptyset \neq  \bigcap_{i<\lambda} \overline{U}_i \subseteq \bigcap_{i<\lambda}D_i\cap O$.

\end{proof}

\begin{corollary} \label{Corollary:UncountableVersionOfBaireCategoryTheoremForAnIntersectionOfOpenSubclasses}
Let $\lambda$ be an infinite cardinal. If $X$ is locally compact regular topological space and $\{G_{i}\}_{i < \lambda}$ is a collection of $\lambda$-open subclasses of $X$, then $\bigcap_{i < \lambda} G_{i}$ has the $\lambda$-Baire property.
\end{corollary}

\begin{proof}
Similar to the proof of Proposition~\ref{Proposition:UncountableVersionOfBaireCategoryTheorem}, but with the added condition that $U_{i} \subseteq G_{i}$ for every $i < \lambda$.
\end{proof}

\subsection{Statement of the general omitting types theorem}

Let $\mathcal{L}$ be a logic and let $\mathbf S$ be a signature.  
If $\bar{x} = x_{1},\dots,x_{n}$ a finite list of constant symbols not in $S$ and $\Sigma$ is a set of $(S \cup \{\bar{x}\})$-sentences, we emphasize this by writing $\Sigma$ as $\Sigma(\bar{x})$.  If $M$ is an $S$-structure and $\bar{a} = a_{1},\dots,a_{n}$ is a tuple of elements of $M$ such that $M\models_{\mathcal{L}} \sigma[\bar{a}]$ for every $\sigma(\bar x)\in\Sigma(\bar x)$, we say that $M$ \emph{satisfies} $\Sigma(\bar x)$, and that $\bar a$ \emph{realizes} $\Sigma(\bar x)$ in $M$. If $T$ is an $S$-theory and $M$ is an $\mathbf{S}$-structure that satisfies $T$, we say that $\Sigma(\bar x)$ is an \emph{$\mathbf{S}$-type} of $T$.

If $\Gamma(\bar x),\Sigma(\bar x)$ are $\mathbf{S}$-types for a theory $T$, we write $T, \Gamma(\bar x) \models_{\mathcal{L},\mathbf  S} \Sigma(\bar x)$ if whenever $M$ is an $\mathbf{S}$-structure that is a model of $T$, every realization of $\Gamma(\bar x)$ in $M$ is also a realization of $\Sigma(\bar x)$.

\begin{definition}
Let $\Sigma(\bar{x})$ be an $\mathbf{S}$-type of a consistent theory $T$.  A set $\Phi(\bar{x})$ of $S$-formulas is a \emph{generator} of $\Sigma(\bar{x})$ over $T$ (or \emph{generates} $\Sigma(\bar{x})$ over $T$) if
\begin{enumerate}
\item $T \cup \Phi(\bar{x})$ is satisfiable by an $\mathbf{S}$-structure, and
\item $T\cup\Phi(\bar{x}) \models_{\mathcal{L},\mathbf{S}} \Sigma(\bar{x})$.
\end{enumerate}
\end{definition}

\begin{definition} \label{Definition:PrincipalType}
Let $T$ be a consistent $S$-theory and let $\Sigma(\bar{x})$ be an $\mathbf{S}$-type of $T$.
\begin{enumerate}
\item
If $\lambda$ is an uncountable cardinal, we will say that  $\Sigma(\bar{x})$ is \emph{$\lambda$-principal}\index{type!lambda-principal@$\lambda$-principal|ii} over $T$ if  there exists a set of cardinality less than $\lambda$ that generates $\Sigma(\bar{x})$ over $T$.
\item
We will say that a $\Sigma(\bar{x})$ is \emph{$\omega$-principal} over $T$ if there exist terms $t_1(\bar y),\dots,t_n(\bar y)$, where $n=\ell(\bar x)$ and a single formula   $\phi(\bar y)$ such that
\begin{enumerate}
\item
$\phi(\bar y)$ generates the type $\Sigma(t_1(\bar y),\dots,t_n(\bar y))$ over $T$, and
\item
$T\cup\{\phi(\bar{y})\ge r\} \models_{\mathcal{L},\mathbf{S}}\Sigma(t_1(\bar y),\dots,t_n(\bar y))$, for some $r\in\mathbb{Q}\cap(0,1)$.
\end{enumerate}
\end{enumerate}
\end{definition}


If  $\Sigma(\bar{x})$  is an $\mathbf{S}$-type for $T$ and $M$ is a $\mathbf{S}$-structure that is a model of $T$, we will say that $M$ \emph{omits} $\Sigma(\bar{x})$ if it does not realize it, i.e., if for every $\bar{a} \in M$, there is $\sigma(\bar{x}) \in \Sigma(\bar{x})$ such that $M \models_{\mathcal{L}} \sigma[\bar{a}] < 1$.

\begin{definition}  \label{Definition:LambdaOmittingTypesProperty}
Let $\mathcal{L}$ be a logic and $\lambda$ an infinite cardinal.  We will say that $\mathcal{L}$ has the \emph{$\lambda$-Omitting Types Property}\index{lambda-Omitting Types Property@$\lambda$-Omitting Types Property|ii} if whenever $T$ is an $S$-theory of cardinality $\leq \lambda$ that is satisfied by an $\mathbf{S}$-structure and  $\{\Sigma_{j}(\bar{x})\}_{j < \lambda}$ is a set of $\mathbf{S}$-types that are not $\lambda$-principal over $T$ there is a model of $T$ that omits each $\Sigma_{j}(\bar{x})$.
\end{definition}

The following is the main result of this section.
\begin{theorem} \label{Theorem:GeneralOmittingTypesTheorem}
Basic continuous logic has the $\lambda$-Omitting Types Property for every infinite cardinal $\lambda$.
\end{theorem}

The rest of this section falls into into two parts.  In the first part, we study the link between our uncountable version of the Baire Category Theorem and the logical topology for regular logics.  In the second part we prove Theorem~\ref{Theorem:GeneralOmittingTypesTheorem}.

\subsection{The $\lambda$-Baire property and classes of structures} 
\begin{definition}\label{D:locally compact logic}
A logic $\mathcal L$ is \emph{locally compact} if  the space $(\,\Str_{\mathcal L}(\mathbf{S}),\tau_{\mathcal{L}}(\mathbf{S})\,)$ is locally compact for every signature $\mathbf{S}$, and \emph{regular} if $(\,\Str_{\mathcal L}(\mathbf{S}),\tau_{\mathcal{L}}(\mathbf{S})\,)$ is regular for every signature $\mathbf{S}$.
\end{definition}

By Proposition~\ref{P:real-valued logic continuity regularity}-\ref{I:real-valued logic regularity}, every $[0,1]$-logic that is closed under the basic connectives is regular.

\begin{proposition} \label{Proposition:TheoryHasBaireProperty}
Let $\mathcal L$ be a locally compact regular logic and let $\lambda$ be an infinite cardinal. Then
\begin{enumerate}
\renewcommand{\theenumi}{\textup{(\arabic{enumi})}}
\renewcommand{\labelenumi}{\theenumi}
\item
$\Str_{\mathcal L}(\mathbf{S})$ has the $\lambda$-Baire Property.
\item
If $T$ is an $S$-theory, then $\Mod_{\mathbf{S}}(T)$ has the $\lambda$-Baire Property. 
\end{enumerate}
\end{proposition}

\begin{proof}
The first part is given by Proposition~\ref{Proposition:UncountableVersionOfBaireCategoryTheorem}. For the second part, notice that since  $\Mod_{\mathbf{S}}(T)$ is closed, it inherits local compactness from $\Str_{\mathcal L}(\mathbf{S})$, so this part follows from Proposition~\ref{Proposition:UncountableVersionOfBaireCategoryTheorem} as well.
\end{proof}

For the rest of this section, the background logic is basic continuous logic, and the background topology is the logical topology on  $\Str_{\mathcal L}(\mathbf{S})$, where $\mathbf{S}$ is a fixed signature of cardinality~$\leq \lambda$

Let $C = (c_{i})_{i < \lambda}$ be a family of new constants. We denote by $(\mathbf{S} \cup C)$ the signature that results from adding the constants in $C$ to $\mathbf{S}$.  We will denote $(\mathbf{S}\cup C)$-structures as $(M,\vec{a})$, where $M$ is an $\mathbf{S}$-structure and $\vec{a} = (a_{i})_{i < \lambda}$ interprets $(c_{i})_{i < \lambda}$.   In this context, $\langle \vec{a}\rangle$ denotes the closure of $\vec{a}$ under the functions of $M$, and $M\upharpoonright \langle \vec{a}\rangle$ denotes the substructure of $M$ induced by $\langle \vec{a}\rangle$.

For the rest of the section, $\mathcal{W}$ will denote the class of $(\mathbf{S}\cup C)$-structures of the form $(M,\vec{a})$ such that $M \upharpoonright \langle\vec{a}\rangle$ is an elementary substructure of $M$ and $\langle\vec{a}\rangle$ is contained in the topological closure of $\vec{a}$ in $M$.

\begin{proposition} \label{Proposition:ClassOfStructuresWithWitnessesHasBaireProperty}
Let $\lambda$ be an infinite cardinal and let $T$ be a $(S \cup C)$-theory of cardinality at most $\lambda$ that is satisfied by an $\mathbf{S}$-structure. Then the class $\mathcal{W} \cap\Mod_{\mathbf{S}}(T)$ is nonempty and has the $\lambda$-Baire Property.
\end{proposition}

\begin{proof}
Let $\big( \varphi_{i}(x) \big)_{i < \lambda}$ be a list of all the $(S \cup C)$-formulas in one variable.  By the Tarski-Vaught Test (Theorem~\ref{Theorem:TarskiVaughtTestForContinuousLogic}), we have $(M,\vec{a}) \in \mathcal{W}$ if and only if the following condition holds:
\[
	\parbox{300pt}{for every $i < \lambda$ and $r \in \mathbb{Q}\cap(0,1)$, if $M \models_{\cl} \exists x \, \varphi_{i}(x)$, then $M\models_{\cl} \varphi_{i}(c_{j}) > r$ for some $j < \lambda$.}
\]
This means that
\[
	\mathcal{W} = \bigcap_{i < \lambda} \; \bigcap_{r \in \mathbb{Q} \cap (0,1)} \bigg( \Mod_{\mathbf{S}}\big( \exists x \, \varphi_{i}(x) < 1 \big) \cup \bigcup_{j < \lambda} \Mod_{\mathbf{S}} \big( \varphi_{i}(c_{j}) > r \big) \bigg).
\]
Thus, by Corollary~\ref{Corollary:UncountableVersionOfBaireCategoryTheoremForAnIntersectionOfOpenSubclasses}, the class $\mathcal{W}$ has the $\lambda$-Baire property.  The class  $\mathcal{W} \cap \Mod_{\mathbf{S}}(T)$ has the  $\lambda$-Baire property by  Corollary~\ref{Corollary:UncountableVersionOfBaireCategoryTheoremForAnIntersectionOfOpenSubclasses} and Proposition~\ref{Proposition:TheoryHasBaireProperty}.  This class is non empty because, by the downward Lšwenheim-Skolem-Tarski Theorem (Theorem~\ref{Theorem:DownwardLowenheimSkolemTheoremForContinuousLogic}),  $T$ is satisfied by an $\mathbf{S}$-structure $M$ of cardinality~$\le\lambda$, and thus interpreting the constants by an enumeration
$\vec{a}$ of $M$ yields $(M,\vec{a})\in\mathcal{W}$.
\end{proof}

\subsection{Proof of the general omitting types theorem} 

Note that if $\lambda$ is an uncountable cardinal and $\Sigma(\bar x)$ is $\lambda$-principal over $T$, where $\bar x=x_1,\dots, x_n$ and $t_1(\bar y),\dots,t_n(\bar y)$ are terms, then the type $\Sigma(t_1(\bar y),\dots,t_n(\bar y))$ is  $\lambda$-principal over $T$. The following lemma shows that in basic continuous logic the converse of this fact holds.

\begin{lemma}\label{L:principal type term substitution}
If $\lambda$ is an uncountable cardinal  and $\Sigma(x_1,\dots,x_n)$ is an $\mathbf{S}$-type of  $T$ such that $\Sigma(t_1(\bar y),\dots,t_n(\bar y))$ is $\lambda$-principal over $T$, where $t_1(\bar y),\dots,t_n(\bar y)$ are $S$-terms, then $\Sigma(x_1,\dots,x_n)$ is  $\lambda$-principal over $T$.
\end{lemma}

\begin{proof}
Let $\{\phi_i(\bar y)\}_{i\in I}$ generate $\Sigma(t_1(\bar y),\dots,t_n(\bar y)\,\big)$ over $T$. 
For each finite $I_0\subseteq I$,  define $\psi_{I_0}(\bar x)$ as
\[
\exists \bar y \big(\,\bigwedge_{k\le n} d(x_k,t_k(\bar y))\le 0\land \bigwedge_{i\in I_0} \phi_i(\bar y)\,\big).
\]
We claim that the set of formulas of the form  $\psi_{I_0}(\bar x)$, where $I_0$ is a finite subset of $I$, generates  $\Sigma(\bar x)$ over $T$. The proof is a standard compactness argument in basic continuous logic, but since we are not assuming previous experience with continuous logic, we include the details below. 

Fix $\sigma(\bar x)\in\Sigma(\bar x)$ and rationals $r,r'$ such that $0<r<r'<1$. Since
$T\cup\{\phi_i(\bar y)\}_{i\in I}\models_{\cl,\mathbf{S}}\Sigma(t_1(\bar y),\dots,t_n(\bar y))$, by  compactness there exists a finite $I_0\subseteq I$ such that 
\[
\tag{*}
T\cup\big\{\,\bigwedge_{i\in I_0}\phi_i(\bar y)\,\big\}
\models_{\cl,\mathbf{S}}\sigma(t_1(\bar y),\dots,t_n(\bar y))\ge r.
\]
Since $\mathbf{S}$ includes uniform continuity moduli for all the predicate and operation symbols that occur in $ \sigma$, there exists  a rational $\delta\in(0,1)$ such that 
\[
\tag{**}
\bigwedge_{k\le n} d(x_k,t_k(\bar y))\le\delta \land \sigma(t_1(\bar y),\dots,t_n(\bar y))\ge r'
\models_{\cl,\mathbf{S}}\sigma(x_1,\dots,x_n)\ge r.
\]
By $(*)$ and $(**)$,
\[
T\cup\big\{\bigwedge_{k\le n} d(x_k,t_k(\bar y))\le\delta\land \bigwedge_{i\in I_0}\phi_i(\bar y)\,\big\}
\models_{\cl,\mathbf{S}}\sigma(x_1,\dots,x_n)\ge r.
\]
Since $\sigma$ is arbitrary and $r$ is arbitrarily close to 1, this shows that the set of formulas of the form   $\psi_J(\bar x)$, where $J$ is a finite subset of $I$,  generates $\Sigma(\bar x)$ over $T$.
\end{proof}

We now prove some lemmas that connect principality with the logical topology  $\tau_{\mathcal{L}}$ on $\Str_{\mathcal L}(\mathbf{S})$ (see Definition~\ref{Definition:LogicalTopology}).

Recall that if $X$ is a topological space and $x\in X$, then a $\lambda$-neighborhood of $x$ is an intersection of less than $\lambda$ neighborhoods of $x$. If $A$ is a subclass of $X$, the $\lambda$-interior of $A$ is the set of points in $x$ that have a $\lambda$-neighborhood contained in $A$.

\begin{lemma} \label{Lemma:EquivalenceOfPrincipalAndNonEmptyInterior}
Let $\Sigma (\bar{x})$ be an $\mathbf{S}$-type of a theory $T$ with $\ell(\bar x)=n$, and let $\lambda$ be an infinite cardinal.
\begin{enumerate}
\renewcommand{\theenumi}{\textup{(\arabic{enumi})}}
\renewcommand{\labelenumi}{\theenumi}
\item
If $\lambda $ is uncountable, then $\Sigma (\bar{x})$ is $\lambda$-principal over $T$ if and only if the class $\Mod_{\mathbf{S}\cup \{\bar{x}\}}\big(T\cup \Sigma (\bar{x})\big)$ has nonempty $\lambda $-interior in $\Mod_{\mathbf{S}\cup \{\bar{x}\}}(T)$. 
\item
If $\lambda =\omega$, then $\Sigma (\bar{x})$ is $\omega $-principal over $T$ if
and only if there exist terms $t_1(\bar y),\dots,t_n(\bar y)$ such that the class $\Mod_{\mathbf{S}\cup \{\bar{y}\}}\big(T\cup \Sigma (t_{1}(\bar{y}),\dots,t_n(\bar{y}))\big)$ has nonempty interior in $\Mod_{\mathbf{S}\cup \{\bar{y}\}}(T)$.
\end{enumerate}
\end{lemma}

\begin{proof}
Assume that $\lambda $ is uncountable and $\Sigma (\bar{x})$ is $
\lambda $-principal over $T$, and let $\{\varphi _{i}(\bar{x})\}
_{i<\mu }$ generate $\Sigma (\bar{x})$ over $T$, where $\mu <\lambda$. The class
\[
\Mod_{\mathbf{S}\cup \{\bar{x}\}}(\varphi _{i}(\bar{x}))=
\bigcap _{r\in \mathbb{Q}\cap(0,1)}\Mod_{\mathbf{S}\cup \{\bar{x}\}}(\varphi _{i}(\bar{x})>r)
\]
is $\lambda$-open and hence so is 
\[
\Mod_{\mathbf{S}\cup \{\bar{x}\}}(\{\varphi _{i}(\bar{x})\}_{i<\mu })=\bigcap _{i<\mu }\Mod_{\mathbf{S}\cup \{\bar{x}\}}(\varphi _{i}(\bar{x})).
\]
Since $T\cup \{\varphi _{i}(\bar{x})\}_{i<\mu }$ is satisfiable by an $\mathbf{S}$-structure and $T\cup\{\varphi _{i}(\bar{x})\}_{i<\mu }\models _{\cl\mathbf{S}}\Sigma (\bar{x})$, the class
\[
\Mod_{\mathbf{S}\cup \{\bar{x}\}}(T)\cap\Mod_{\mathbf{S}\cup \{\bar{x}\}}((\{\varphi _{i}(\bar{x})\}_{i<\mu })
\]
is a nonempty $\lambda $-open subclass of $\Mod_{\mathbf{S}\cup \{\bar{x}\}}(T)$ contained in $\Mod_{\mathbf{S}\cup \{\bar{x}\}}(\Sigma (\bar{x}))$.

Suppose, conversely, that $\Mod_{\mathbf{S}\cup \{\bar{x}\}}(T\cup \Sigma (\bar{x}))$ has nonempty $\lambda $-interior in $\Mod_{\mathbf{S}\cup \{\bar{x}\}}(T)$. Then there exist $\mu <\lambda $ and  formulas $\varphi _{i}(\bar{x})$ for $i<\mu $ such that 
\begin{equation*}
\Mod_{\mathbf{S}\cup \{\bar{x}\}}(T)\cap \bigcap_{i<\mu }\Mod_{\mathbf{S}\cup \{\bar{x}\}}(\varphi _{i}(\bar{x})>0)
\end{equation*}
is a nonempty subclass of $\Mod_{\mathbf{S}\cup \{\bar{x}\}}(T\cup \Sigma (\bar{x}))$. Choose $r_{i}\in \mathbb{Q}\cap(0,1)$ such that $T\cup \{\varphi _{i}(\bar{x})\geq r_{i}\}_{i<\mu }$ is satisfiable by an $\mathbf{S}$-structure. Then, 
\begin{equation*}
T\cup \{\varphi _{i}(\bar{x})\geq r_{i}\}_{i<\mu }\models _{\cl,\mathbf{S}}\Sigma (\bar{x}),
\end{equation*}
which finishes the proof of the uncountable case.

For the case $\lambda =\omega$, suppose that $\varphi(\bar y)$ generates $\Sigma (t_{1}(\overline{y}),\dots,t_n(\overline{y}))$ over $T$ and $T\cup\{\phi(\bar{y})\ge r\} \models_{\mathcal{L},\mathbf{S}}\Sigma(t_1(\bar y),\dots,t_n(\bar y))$, where $t_{1}(\overline{y}),\dots,t_n(\overline{y})$ are terms and $r\in\mathbb{Q}\cap(0,1)$. Since, by the definition of generator, $\Mod_{\mathbf{S}\cup \{\bar{y}\}}(T\cup\{\phi(\bar{y})\})$ is nonempty, if $r'\in\mathbb{Q}\cap(r,1)$, the class $\Mod_{\mathbf{S}\cup \{\bar{y}\}}(T)\cap\Mod_{\mathbf{S}\cup \{\bar{y}\}}(\varphi (\bar{y})>r^{\prime })$
is a nonempty open subclass of $\Mod_{\mathbf{S}\cup \{\bar{y}\}}(T)$ contained in  $\Mod_{\mathbf{S}\cup \{\bar{y}\}}(\Sigma
(t_{1}(\bar{y}),\dots,t_{n }(\bar{y}))).$

To finish the proof for the countable case, assume that $\Mod_{\mathbf{S}\cup \{\bar{y}\}}(T\cup \Sigma (t_{1}(\bar{y}),\dots,t_{n }(\bar{y})))$ has nonempty interior. Then there is a formula $\phi(\bar y)$ such that
\[
\emptyset\neq\Mod_{\mathbf{S}\cup \{\bar{y}\}}(T)\cap
\Mod_{\mathbf{S}\cup \{\bar{y}\}}(\varphi(\bar{y})>0)\subseteq 
\Mod_{\mathbf{S}\cup \{\bar{y}\}}(T\cup \Sigma (t_{1}(\bar{y}),\dots,t_{n }(\bar{y}))).
\]
As before, choose $r\in \mathbb{Q}\cap(0,1)$ such that $T\cup \{\varphi(\bar{y})\geq r\}$ is satisfiable by an $\mathbf{S}$-structure. Then,
\begin{equation*}
T\cup \{\varphi(\bar{y})\geq r\}\models _{\cl,\mathbf{S}}
T\cup \{\varphi(\bar{y})>0\}\models _{\cl,\mathbf{S}}\Sigma (t_{1}(\bar{y}),\dots,t_{n }(\bar{y})).
\end{equation*}%
This shows that  $\Sigma (\bar{x})$ is $\omega$-principal over $T$. 
\end{proof}

For the next lemma, we follow the notation used in Proposition~\ref{Proposition:ClassOfStructuresWithWitnessesHasBaireProperty}. Let $T$ be an $S$-theory.  For $\textbf{\textup{i}} = i_{1},\dots,i_{n} \in \lambda$, let $R_{T,\textbf{i}}$ be the map
\begin{align*}
	R_{T,\textbf{\textup{i}}} : \mathcal{W} \cap \Mod_{\mathbf{S} \cup C}(T) &\to \Mod_{\mathbf{S} \cup \{c_{i_{1}},\dots,c_{i_{n}}\}}(T) \\
	(M,\vec{a}) &\mapsto (M,a_{i_{1}},\dots,a_{i_{n}}).
\end{align*}

\begin{lemma} \label{Lemma:ProjectionsAreContinuousOpenAndxive}
The map $R_{T,\textbf{\textup{i}}}$ is continuous, open and surjective.
\end{lemma}

\begin{proof}
For notational simplicity, we will consider the case when $n = 1$.  Continuity follows directly from the fact that any $(S \cup \{c_{i_{1}}\})$-sentence is also an $(S \cup C)$-sentence.

Now, with each $(S \cup C)$-sentence $\varphi = \varphi(c_{1},\dots,c_{i_{1}},\dots,c_{m})$, with all the constants in $C$ exhibited, let us associate the $(S \cup \{c_{i_{1}}\})$-sentence $\theta(c_{i_{1}})$ defined as
\[
	\forall x_{1} \cdots \forall x_{m} \, \varphi(x_{1},\dots,c_{i_{1}},\dots,x_{m}).
\]
To show that $R_{T,\textbf{i}}$ is open and surjective, it suffices to show that $R_{T,\textbf{i}}$ maps $\Mod_{\mathbf{S}\cup C}(\varphi)^{c} \cap \mathcal{W}\cap\Mod_{\mathbf{S} \cup C}(T)$ onto $\Mod_{\mathbf{S}\cup\{c_{i_1}\}} \big( \theta(c_{i_{1}}) \big)^{c}\cap\Mod_{\mathbf{S} \cup \{c_{i_{1}}\}}(T)$.

Suppose $(M,\vec{a}) \in\Mod_{\mathbf{S}\cup C}(\varphi)^{c} \cap \mathcal{W}\cap\Mod_{\mathbf{S} \cup C}(T)$, and let $r \in \mathbb{Q}\cap(0,1)$ be such that $(M,\vec{a}) \models_{\cl} \varphi \leq r$.  We certainly have $(M,a_{i_{1}}) \models_{\cl} \theta(c_{i_{1}}) \leq r$, so $(M,a_{i_{1}}) \in \Mod_{\mathbf{S}\cup\{c_{i_1}\}} \big( \theta(c_{i_{1}}) \big)^{c}\cap\Mod_{\mathbf{S} \cup \{c_{i_{1}}\}}(T)$.  

Suppose, conversely, that $(M,a_{i_{1}}) \in \Mod_{\mathbf{S}\cup\{c_{i_1}\}} \big( \theta(c_{i_{1}}) \big)^{c}\cap\Mod_{\mathbf{S} \cup \{c_{i_{1}}\}}(T)$, and let $r \in \mathbb{Q}\cap(0,1)$ be such that $(M,a_{i_{1}}) \models_{\cl} \theta(c_{i_{1}}) \leq r$.  Pick $r'\in\mathbb{Q}\cap(r,1)$. Then there are elements $a_{k} \in M$, for $k \leq m$ and $k \neq i_{1}$, such that $(M,a_{1},\dots,a_{m}) \models_{\cl} \varphi \leq r'$.  Since $|T|\le \lambda$, the downward L\"{o}wenheim-Skolem Theorem (Theorem~\ref{Theorem:DownwardLowenheimSkolemTheoremForContinuousLogic}) guarantees that there is an elementary substructure $M_{0}$ of $M$ of cardinality~$\le\lambda$ containing $a_{1},\dots,a_{m}$.  Using the constants  $c_j$ with $j\notin\{i_1,\dots,i_n\}$ to name the remaining elements of $M_{0}$, we see that $(M,\vec{a}) \in \Mod_{\mathbf{S}\cup C}(\varphi)^{c} \cap \mathcal{W}\cap\Mod_{\mathbf{S} \cup C}(T)$.
\end{proof}

We now have the material we need to prove Theorem~\ref{Theorem:GeneralOmittingTypesTheorem}:

\begin{proof}[Proof of the omitting types theorem]
Let $T$ be theory of cardinality $\leq \lambda$ that is satisfied by an $\mathbf{S}$-structure and let $\{\Sigma_{j}(x_1,\dots,x_n)\}_{j < \lambda}$ be a set of types that are not $\lambda$-principal over $T$. By hypothesis in the countable case, and by Lemma~\ref{L:principal type term substitution} in the uncountable case, the types $\Sigma_{j}(t_{1}(\bar{y}),\dots,t_n(\bar{y}))$ are not principal over $T$, for any choice of terms $t_{1}(\bar{y}),\dots,t_n(\bar{y})$; hence,  without loss of generality, we may assume that  
\[\tag{*}
\text{$\Sigma(t_{1}(\bar{y}),\dots,t_n(\bar{y}))$ is on the list whenever $\Sigma(x_1,\dots,x_n)$ is.} 
\]
Let $C$ be as before. By Lemma~\ref{Lemma:EquivalenceOfPrincipalAndNonEmptyInterior}, for any $\textbf{i} = i_{1},\dots,i_{n}$ and any $j < \lambda$, the class
\[
	\Mod_{\mathbf{S} \cup \{c_{i_{1}},\dots,c_{i_{n}}\}} \big( T \cup \Sigma_{j}(c_{i_{1}},\dots,c_{i_{n}}) \big)
\]
is closed with empty $\lambda$-interior.  Therefore, by Lemma~\ref{Lemma:ProjectionsAreContinuousOpenAndxive}, so is
\[
	\mathcal{C}_{T,\textbf{i}} = R_{T,\textbf{i}}^{-1} \big( \Mod_{\mathbf{S} \cup \{c_{i_{1}},\dots,c_{i_{n}}\}} \big( T \cup \Sigma_{j}(c_{i_{1}},\dots,c_{i_{n}}) \big) \big).
\]
Since $\mathcal{W} \cap \Mod_{S \cup C}(T)$ is nonempty and has the $\lambda$-Baire Property, there is 
\[
(M,\vec{a}) \in \mathcal{W} \cap \Mod_{\mathbf{S} \cup C}(T) \setminus \bigcup_{\textbf{i}} \mathcal{C}_{T,\textbf{i}}.
\]
Thus, for any $j < \lambda$,  no subset of $\vec{a}$ realizes $\Sigma_{j}(\bar{x})$ in $M$. Furthermore, by our assumption~$(*)$ above, no subset of $\langle \vec{a}\rangle$ realizes $\Sigma_{j}(\bar{x})$ in $M$. This means that $M\upharpoonright \langle \vec{a}\rangle $ omits each $\Sigma _{j}(\bar{x})$. The structure  $M\upharpoonright \langle \vec{a}\rangle $ is a model of $T$ because $M\upharpoonright \langle \vec{a}\rangle \prec _{\cl}M$, since $M\in\mathcal{W}$.
\end{proof}

\subsection{Omitting types in complete structures}
\label{S:Omitting types in complete structures}

Here, the background logic is basic  continuous logic, and $\mathbf{S}$ denotes a fixed signature with vocabulary $S$.

If $\Sigma(x_1,\dots,x_n)$ is a type and $\delta\in\mathbb{Q}\cap[0,1]$, we denote by $\Sigma^\delta(x_1,\dots,x_n)$ the type consisting of all the formulas of the form
\[
\exists y_1\dots\exists y_n\,
\big(\, \bigwedge_{k\le n} d(x_k,y_k)\le \delta \land \sigma(y_1,\dots,y_n) \,\big),
\]
where $\sigma$ ranges over all finite conjunctions of formulas in $\Sigma$.

Note that if $\bar a=a_1,\dots,a_n$ realizes $\Sigma$ in a structure $M$, then every point in the closed $\delta$-ball of $\bar a$ realizes $\Sigma^\delta$, that is, if $\bar b=b_1\dots,b_n\in M$ is such that $\max_{k\le n} d(a_k,b_k)\le \delta$, then $\bar b$ realizes $\Sigma^\delta$. 

\begin{definition} 
\label{D:metrically principal types}
Let $T$ be a consistent $S$-theory and let $\Sigma(\bar{x})$ be an $\mathbf{S}$-type of $T$.
\begin{enumerate}
\item
If $\lambda$ is an uncountable cardinal, we will say that  $\Sigma(\bar{x})$ is \emph{metrically $\lambda$-principal} over $T$ if  $\Sigma^\delta$ is $\lambda$-principal over $T$ for every $\delta>0$.
\item
We will say that a $\Sigma$ is \emph{metrically $\omega$-principal} over $T$ if for every $\delta>0$ there is a formula $\phi(\bar x)$ such that
\begin{enumerate}
\item
$\phi(\bar x)$ generates  $\Sigma^\delta(\bar x)$ over $T$, and
\item
$T\cup\{\phi(\bar{y})\ge r\} \models_{\mathcal{L},\mathbf{S}}\Sigma^\delta(\bar x)$, for some $r\in\mathbb{Q}\cap(0,1)$.
\end{enumerate}
\end{enumerate}
\end{definition}

\begin{proposition}
 \label{P:metrically principal types}
Let $T$ be a consistent $S$-theory and let $\Sigma(\bar x)$ be an $\mathbf{S}$-type of $T$.
Then, for every infinite cardinal $\lambda$, the type $\Sigma$ is metrically $\lambda$-principal over $T$ if and only if $\Sigma^\delta$ is $\lambda$-principal over $T$ for every $\delta>0$.
\end{proposition}

\begin{proof}
For uncountable $\lambda$ the proposition is true by definition.  For $\lambda=\omega$, we only need to prove that if $\Sigma(x_1,\dots,x_n)$ is an $\mathbf{S}$-type of  $T$ such that $\Sigma(t_1(\bar y),\dots,t_n(\bar y))$ is metrically $\omega$-principal over $T$, where $t_1(\bar y),\dots,t_n(\bar y)$ are $S$-terms, then $\Sigma(x_1,\dots,x_n)$ is  metrically $\omega$-principal over $T$. Using the uniform continuity moduli of $\mathbf{S}$ as in the proof of Lemma~\ref{L:principal type term substitution} one can see that for every $\epsilon\in\mathbb{Q}\cap(0,1)$ there exist $\delta,\rho\in\mathbb{Q}\cap(0,1)$ such that if $\phi(\bar x)$ generates  $\Sigma^\delta(\bar x)$ over $T$, and
\[
T\cup\{\phi(\bar{y})\ge r\} \models_{\mathcal{L},\mathbf{S}}\Sigma^\delta(t_1(\bar y),\dots,t_n(\bar y)),
\]
then
\[
\exists \bar u \big(\,\bigwedge_{k\le n} d(x_k,t_k(\bar u))\le 0\land \phi(\bar u)\,\big) \ge \rho
 \models_{\mathcal{L},\mathbf{S}}\Sigma^\epsilon(x,\dots,x).
\]

\end{proof}

\begin{proposition}[$\lambda$-Omitting Types Property for Complete Structures]
\label{P:omitting types in complete structures}
 Let $\lambda$ be any infinite cardinal. If $T$ is an $S$-theory of cardinality $\leq \lambda$ that is satisfied by an $\mathbf{S}$-structure and  $\{\Sigma_{j}(\bar{x})\}_{j < \lambda}$ is a set of $\mathbf{S}$-types that are not metrically $\lambda$-principal over $T$ there is a model of $T$ of cardinality~$\le\lambda$ whose metric completion omits each $\Sigma_{j}(\bar{x})$.
\end{proposition}

\begin{proof}
Immediate from Proposition~\ref{P:metrically principal types} an the $\lambda$-Omitting Types Property of $\mathcal{L}$.
\end{proof}

An $S$-theory $T$ is \emph{complete} if for every $S$-sentence $\phi$ and every $r\in\mathbb{Q}\cap(0,1)$, either $\phi\le r$ or $\phi\ge r$ is in $T$. Note that this yields a concept of complete type, since, in our context, types are theories.

 If $T$ is a complete $S$-theory, we can define a topology on the set of all complete $\mathbf S$-types of $T$, as follows. If $p(\bar x),q(\bar x)$ are such types, where $\bar x=x_1\dots, x_n$, we define $d(p,q)$ as the infimum the set of real numbers $r$ such that there exist a model $M$ of $T$ and tuples $\bar{a},\bar{b}\in M$ satisfying $\max_{k\le n} d(a_k, b_k)\le r$. A compactness argument shows that $d$ is a metric. Note that $d(p,q)\le \delta$ if and only if $p^\delta\subseteq q$, where $p^\delta$ is defined as above. Hence, if $p(\bar x)$ is metrically principal, $\delta>0$, and $M$ is an $\mathbf{S}$-structure such that $M\models_{\cl} T$, then there exists $q(\bar x)$ such that $d(p,q)\le \delta$ and $q$ is realized in $M$. We use this to prove the following observation, due Henson:

\begin{proposition}
\label{P:principal types realized}
If $T$ is a complete $S$-theory and $M$ is a complete $\mathbf S$-structure such that $M\models_{\cl} T$, then every complete $\mathbf S$-type of $T$ that is metrically principal is realized in $M$.
\end{proposition}

\begin{proof}
Fix $T$ and $M$ as in the statement of the proposition, and let $p$ be a complete  $\mathbf S$-type of $T$ that is metrically principal. Using compactness and the preceding observations, we find, inductively, a sequence $(q_n)_{n<\omega}$ of complete $\mathbf S$-types for $T$, a chain $M=M_0\prec_{\cl} M_1\prec_{\cl}M_2\prec_{\cl}\dots$ of models of $T$, and  sequences $(\bar a_n)_{n<\omega}$, $(\bar b_n)_{n<\omega}$ such that for every $n<\omega$,
\begin{enumerate}
\item
$q_n\supseteq p^{2^{-n}}$ (so $d(p,q_n)\le 2^{-n}$),
\item
$\bar a_n$ realizes $q_n$ in $M$,
\item
$\bar b_n$ realizes $p$ in $M_{n+1}$,
\item
$d(a_n,b_n)\le 2^{-n}$,
\item
$d(a_{n+1}, b_n)\le 2^{-(n+1)}$.
\end{enumerate}
By (4) and (5) the sequences $(\bar a_n)_{n<\omega}$ and $(\bar b_n)_{n<\omega}$ are Cauchy and asymptotically equivalent (in $\bigcup_{n<\omega} M_n$). Since $M$ is complete, their unique limit is in $M$, by (2). By (3), this limit realizes $p$.
\end{proof}

\begin{remark}
\label{R:omitting types in complete structures}
By Propositions~\ref{P:omitting types in complete structures} and ~\ref{P:principal types realized}, if $T$ is a complete $S$-theory and $p$ is a complete $\mathbf{S}$-type of $T$, then $p$ is metrically principal if and only if $p$ is realized in every complete $\mathbf S$-structure that is a model of $T$. Hence,  the special case of Proposition~\ref{P:omitting types in complete structures} when $\lambda=\omega$  is Henson's omitting types theorem for complete metric structures~\cite[Section 12]{Ben-Yaacov-Berenstein-Henson-Usvyatsov:2008}, \cite[Section 1]{Ben-Yaacov-Usvyatsov:2007}.
\end{remark}

\section{The Main Theorem} \label{Section:TheMainTheorem}

Let $\cl$ denote basic continuous logic. We have proved  (Theorem~\ref{Theorem:GeneralOmittingTypesTheorem}) that $\cl$ has the $\kappa$-Omitting Types Property for every infinite cardinal~$\kappa$. It then follows that continuous logic and {\L}ukasziewicz-Pavelka logic have this property as well. In this section we show that the $\kappa$-Omitting Types Property characterizes $\cl$. Analogous characterizations for  continuous logic and {\L}ukasziewicz-Pavelka logic follow as corollaries.

We prove that any $[0,1]$-valued logic $\mathcal{L}$ for continuous metric structures that extends $\cl$ and has the $\kappa$-Omitting Types Property for some uncountable regular cardinal $\kappa$ is equivalent to $\cl$, as long as only vocabularies of cardinality less than $\kappa$ are considered. We will make some natural assumptions about $\mathcal{L}$, namely, 
\begin{enumerate}\label{L:conditions on the logic L}
\item 
Closure under the basic connectives, and under the existential quantifier (Definition~\ref{D:closed under classical quantifiers}),
\item 
The finite occurrence property (Definition~\ref{D:FiniteOccurrenceProperty}),
\item
Relativization to definable families of predicates (Definition~\ref{D:logic permits relativization}),
\item
Every structure is equivalent in $\mathcal L$ to its metric completion  (see Corollary~\ref{C:StructuresAreElementarySubstructuresOfCompletions}).
\end{enumerate}

The following is the main result of the paper:

\begin{theorem}[Main Theorem] \label{Theorem:TheMainTheorem}
Let $\mathcal{L}$ be a $[0,1]$-valued logic that satisfies properties \textup{(1)--(5)} above and has the $\kappa$-Omitting Types Property for some uncountable regular cardinal $\kappa$. If $\mathcal L$ extends $\cl$,  then $\mathcal{L}$ is equivalent to  $\cl$  for signatures of cardinality less than $\kappa$.
\end{theorem}

Fix a cardinal $\kappa$ as given by the statement of Theorem~\ref{Theorem:TheMainTheorem}, let $\mathbf S$ be a signature of cardinality less than $\kappa$, and let us view $S$-sentences a $[0,1]$-valued functions on the class of $S$-structures, by identifying an $S$-sentence $\phi$ with the function $M\mapsto \phi^M$. The equivalence stated by Theorem~\ref{Theorem:TheMainTheorem} means that the logical topologies $\tau_{\mathcal L}$ and $\tau_{\cl}$ on the class of $\mathbf{S}$-structures coincide, so, for every $S$-sentence $\phi$ of $\mathcal{L}$, the class $\Mod_{\mathbf S}(\phi)$ is $\tau_{\cl}$-closed. This implies that every $S$-sentence $\mathcal L$ (viewed as a $[0,1]$ valued function) is  $\tau_{\cl}$-continuous, because if $\phi$ is such a sentence and $[r,s]$ is a subinterval of $[0,1]$ with rational endpoints, then $\phi^{-1}[r,s]=\Mod_{\mathbf S}(\phi\ge r)\cap \Mod_{\mathbf S}(\phi\le s)$. Since  $\cl$ is compact, by the Stone-Weierstrass Theorem for lattices, every sentence of $\mathcal L$ can be approximated,  uniformly over the class of $\mathbf S$-structures, by sentences of $\cl$. 

The rest of this section is devoted to the proof of the Main Theorem. The strategy of the proof  is to show that if $\mathcal{L}$ extends $\cl$ strictly, then there exist structures that are metrically isomorphic, but nonequivalent in $\mathcal{L}$; this contradicts the Isomorphism Property of $\mathcal L$ (see Definition~\ref{Definition:LogicalSystem}).

The section is divided into two parts. In the first one we use the $\kappa$-Omitting Types Property of the logic  $\mathcal L$ to prove that $\mathcal{L}$ is $\lambda$-compact for every  $\lambda < \kappa$ (Proposition~\ref{Proposition:LambdaCompactness}), and the second part of the section is devoted to the proof of the Main Theorem.

\subsection{Obtaining compactness from the Omitting Types Property} 
\label{Subsection:UsingOmittingTypesToObtainCompactness}

Here, $\mathcal{L}$ denotes a  $[0,1]$-valued logic that satisfies the hypotheses of the Main Theorem (Theorem~\ref{Theorem:TheMainTheorem}) and $\kappa$ denotes an uncountable regular cardinal such that $\mathcal{L}$ has the $\kappa$-Omitting Types Property.  

Recall that if $M$ is an $S$-structure, an $[0,1]$-valued predicate $R^{M}(\bar{x})$ is discrete if $R^{M}$ only takes on values in $\{0,1\}$.  As we observed in Definition~\ref{Definition:DiscretePredicate}, if $R$ is a predicate in $S$, then the interpretation $R^{M}$ is discrete if and only if
\[
	M \models_{\mathcal{L}} \forall \bar{x} \, \Discrete \big( R(\bar{x}) \big),
\]
where $\Discrete \big( R(\bar{x}) \big)$ is an abbreviation of the  sentence $R(\bar{x}) \vee \neg R(\bar{x})$.

\begin{definition} \label{D:DescriptionOfDiscreteLinearOrdering}
Let $M$ be a structure, $P$ a new monadic predicate symbol, and $\triangleleft$ a new binary predicate symbol.  We will say that $(P^{M},\triangleleft^{M})$ is a \emph{discrete linear ordering} if $M \models_{\mathcal{L}} \theta$, where $\theta$ is the conjunction of the following sentences:
\begin{itemize}
\item $\forall x \, \Discrete \big( P(x) \big)$

\item $\forall x,y \, \big( \neg P(x) \lor \neg P(y) \lor \Discrete(x \triangleleft y)\big)$

\item $\forall x,y \, \big( \neg P(x) \lor \neg P(y) \lor \Discrete(d(x, y)\big)$

\item $\forall x \, \big[ \neg P(x) \vee \neg (x \triangleleft x) \big]$

\item $\forall x,y \, \big[ \neg \big( P(x) \wedge P(y) \wedge (x \triangleleft y) \big) \vee \neg (y \triangleleft x) \big]$

\item $\forall x,y,z \, \big[ \neg \big( P(x) \wedge P(y) \wedge (x \triangleleft y) \wedge (y \triangleleft z) \big) \vee (x \triangleleft z) \big]$

\item $\forall x,y \, \big[ \neg \big( P(x) \wedge P(y) \big) \vee \big( (x \triangleleft y) \vee (y \triangleleft x) \vee\neg d (x , y) \big]$.
\end{itemize}
\end{definition}

Notice that if $M$ be a structure, $P$ is monadic predicate symbol, and $(P^{M},\triangleleft^{M})$ is a discrete linear ordering, then $(P^{M},\triangleleft^{M})$ is a linear ordering in the usual sense.

\begin{lemma} \label{Lemma:DiscreteLinearOrderingExtension}
Let  $T$ be a consistent $\mathcal{L}$-theory of cardinality $\leq \kappa$.  If $T$ has a model $M$ such that $(P^{M}, \triangleleft^{M})$ is a discrete linear ordering with no right endpoint, then $T$ has a model $N$ such that $(P^{N}, \triangleleft^{N})$ is a discrete linear ordering of cofinality $\kappa$.
\end{lemma}

\begin{proof}
Let $(c_{i})_{i < \kappa}$ be a family of new  constants.  Let $\theta$ be the sentence given in Definition~\ref{D:DescriptionOfDiscreteLinearOrdering}.  Define
\begin{align*}
T'=T \cup \{ \theta\} &\cup\{\, \forall x \,( \neg P(x) \vee \exists y \, ( P(y) \wedge (x \triangleleft y)\,)\,)\,\}\\ 
&\cup\{P(c_{i})\}_{i < \kappa} \cup \{(c_{i} \triangleleft c_{j})\lor \neg d(c_{i},c_{j})\}_{i < j < \kappa}.
\end{align*}
We first claim that $T'$ is satisfiable.  To see this, notice that if $M$ is a model of $T$ such that $(P^{M}, \triangleleft^{M})$ is a discrete linear ordering with no right endpoint, $a \in P^{M}$, and $c_{i}^{M} = a$ for all $i < \kappa$, then $(M,c_{i}^M)_{i < \kappa} \models_{\mathcal{L}} T'$.

Next, we claim that the type
\[
	\Sigma(x) = \{P(x)\} \cup \{\,c_{i} \triangleleft x\lor \neg d(c_i,x)\,\}_{i < \kappa}
\]
is not $\kappa$-principal over $T'$.  Suppose the contrary, and assume that $\Phi(x)$ generates of $\Sigma(x)$ over $T$.   Take a structure $M$ and an element $a$ of $M$ such that $M \models_{\mathcal{L}} T' \cup\Phi[a]$.  If $M \models_{\mathcal{L}} \neg P[a]$, then $M \not\models_{\mathcal{L}} \Sigma[a]$,  so $\Phi(x)$ cannot generate $\Sigma(x)$ over $T$.  Otherwise, since $|\Phi(x)|<\kappa$ and $\mathcal{L}$ has the finite occurrence property, there is $j < \kappa$ such that $c_{j}$ does not occur in $\Phi(x)$. Since $M \models_{\mathcal{L}} \theta' \wedge P[a]$, we can find an interpretation $c_{j}^{M}$ of $c_{j}$ in $M$ such that $a \triangleleft c_{j}^{M}$.  But then $M \not\models_{\mathcal{L}} \Sigma[a]$, so $\Phi(x)$ cannot generate $\Sigma(x)$ over $T$.

Thus $\Sigma(x)$ is not $\kappa$-principal over $T'$.  Since $\mathcal{L}$ has the $\kappa$-Omitting Types Property, there is a model $N$ of $T'$ that omits $\Sigma(x)$.  This means that no element of $P^{N}$ is an upper bound of $\{c_{i}^{N}\}_{i < \kappa}$, i.e., the sequence $(c_{i}^{N})_{i < \kappa}$ is cofinal in $P^{N}$. The result then  follows since $\kappa$ is regular.
\end{proof}

\begin{proposition} \label{Proposition:LambdaCompactness}
The logic $\mathcal{L}$ is $\lambda$-compact for every  $\lambda < \kappa$.
\end{proposition}

\begin{proof}
We prove the proposition by induction on all $\lambda < \kappa$.  Fix $\lambda < \kappa$ and suppose $\mathcal{L}$ is $\mu$-compact for every $\mu < \lambda$.  Let $\mathbf{S}=(S,\mathcal{U})$ be a signature, and let $T = \{\varphi_{i}\}_{i < \lambda}$ be an $\mathcal{L}$-theory such that every finite subset of $T$ is satisfied by an $\mathbf{S}$-structure. We wish to show that $T$is satisfied by an $\mathbf{S}$-structure. 

Let $\mathbf{S_0}$ be the signature $(S_0,\mathcal{U}_0)$, where $S_0$ is the subvocabulary of $S$ formed by the symbols of $S$ that occur in $T$ and $\mathcal{U}_0$ is the restriction of $\mathcal{U}$ to $S_0$.  By the Reduct Property of logics (Definition~\ref{Definition:LogicalSystem}), it suffices to show that $T$is satisfied by an $\mathbf{S}_0$-structure.  By the assumption that $\mathcal{L}$ has the finite occurrence property, we have $|S_0|\le\lambda$, so without loss of generality we can assume that the uniform continuity moduli specified by $\mathcal{U}_0$ are made explicit by sentences in $T$, and hence every $S_0$-structure that satisfies $T$ is bound to be $\mathbf{S}_0$-structure. 

Let $S^+$ be a vocabulary that results from adding to $S$ a unary predicate symbol $P$, two new binary predicate symbols, $R$ and $\triangleleft$, and a family $(c_{i})_{i < \lambda}$ of new constant symbols.  

At this point we invoke the assumption  (given on page~\pageref{L:conditions on the logic L}) that the logic $\mathcal L$ permits relativization to definable families of predicates: for each $S$-sentence $\phi$ and each $i<\lambda$, let $\phi_i^{\{y\mid R(x,y)\}}(x)$ be a relativization of $\phi_i$ to  $\{y\mid R(x,y)\}$ (see Definition~\ref{D:logic permits relativization}).
Define an $S^+$-theory $T^+$ by letting
\begin{align*}
	T^{+} = \{\theta\}&\cup\{P(c_{i})\}_{i < \lambda} \\
	&\cup\big\{\, \forall x \, \big(  \forall y\, ( R(x,y)\lor\neg R(x,y)) \land (\neg (c_{i} \triangleleft x) \vee \varphi_{i}^{\{y\mid R(x,y)\}}(x)) \big) \,\big\}_{i < \lambda},
\end{align*}
where $\theta$ is as in  Definition~\ref{D:DescriptionOfDiscreteLinearOrdering}. We first claim that $T^{+}$ has a model $M$ such that $(P^{M},\triangleleft^{M})$ is a discrete linear ordering with no right endpoint.  By the induction hypothesis, for each $j < \lambda$, the theory $\{\varphi_{i}\}_{i < j}$ has a model $M_{j}$.  Let $M$ be the $S^{+}$-structure defined as follows:
\begin{itemize}
\item
We have $c_i^M\neq c_j^M$ if $i<j<\lambda$, and the universe of $M$ is the disjoint union of $\bigsqcup_{i<\lambda}M_i$  and $\{c^M_{i}\}_{i < \lambda}$.
\item
The distance between pairs of elements of $M$ in the same $M_i$ is as given by the metric of $M_i$, and between pairs of elements of $M$ not in the same $M_i$ it is 1.
\item
If $Q$ is a $n$-ary predicate symbol of $S$, the interpretation $Q^M$ is  $\bigsqcup_{i<\lambda}T^{M_i}$ in $\bigsqcup_{i<\lambda}M_i^n$ and 0 in $M^n\setminus \bigsqcup_{i<\lambda}M_i^n$.
\item
If $f$ is a $n$-ary operation symbol of $S$, and $\bar a \in M^n$, then $f^M(\bar a)$ is  $f^{M_i}(\bar a)$ if $a\in M_i^n$ and $c_0$ if $\bar a\in M^n\setminus \bigsqcup_{i<\lambda}M_i^n$.
\item 
$P^{M}$ is the characteristic function of $\{c^M_{i}\}_{i < \lambda}$.
\item 
$\triangleleft^M$ is characteristic function of $\{\,(c_i^M,c_j^M)\mid i<j\,\}$.
\item
$R^M$ is the characteristic function of $\bigcup_{i<\lambda}\{c_i\}\times M_i$.
\end{itemize}

By the renaming property of $\mathcal L$ (see Definition~\ref{Definition:LogicalSystem}), $M$ is a model of $T^{+}$. Note that $(P^{M},\triangleleft^{M})$ is a discrete linear ordering with no right endpoint. Thus, by Lemma~\ref{Lemma:DiscreteLinearOrderingExtension}, there is a model $N$ of $T^{+}$ such that $(P^{N},\triangleleft^{N})$ is a discrete linear ordering of cofinality $\kappa$.  Since $\lambda < \kappa$, there is $a \in N$ such that $c_{i}^N \trianglelefteq a^N$ for every $i < \lambda$.  Thus, $N \upharpoonright \{\,b\mid N\models_{\mathcal L} R[a,b]\,\}$ is a model of $T$.
\end{proof}

\subsection{Proof of the Main Theorem}
 \label{TheoremProofSection}
Recall that  $\mathcal{L}$ denotes a  $[0,1]$-valued logic that satisfies the hypotheses of the Main Theorem (Theorem~\ref{Theorem:TheMainTheorem}) and $\kappa$ denotes an uncountable regular cardinal such that $\mathcal{L}$ has the $\kappa$-Omitting Types Property.  

Since, by  Proposition~\ref{Proposition:LambdaCompactness},  $\mathcal L$ is $\lambda$-compact for every $\lambda<\kappa$, Proposition~\ref{Proposition:ApproximabilityIndistinguishabilityAndCompactnessImpliesReverseApproximability} provides the following result:

\begin{proposition} \label{P:EquivalenceToContinuousLogic}
Let $\mathbf S$ be a signature with $|S|< \kappa$.  If  $\equiv_{\cl} \To \equiv_{\mathcal{L}}$ for $\mathbf S$-structures, then $\mathcal{L}$ is equivalent to $\cl$ for $\mathbf S$-structures. \hfill $\qed$
\end{proposition}

Thus, all that remains in order to prove the Main Theorem is to show that $\equiv_{\cl} \To \equiv_{\mathcal{L}}$ for signatures of cardinality less than~$\kappa$.

\label{P:amalgamation}
If $S$ is a vocabulary and $M_{0},M_{1}$ are $S$-structures, we form the combined structure $[M_{0},M_{1}]$ in the following way.  For each $n$-ary predicate symbol $R$ of $S$ let $R^0,R^1$ be two distinct $n$-ary predicate symbols and for each $n$-ary operation symbol $f$ of $S$ let $f^0,f^1$ be two distinct $n$-ary operation symbols. Let $P_0,P_1$ be new monadic predicates. For $k=0,1$ let
\[
S^k =\{\, R^k \mid R \text{ in } S\,\}\cup\{\,f^k \mid f\text{ in } S\,\}\cup \{P_k\}.
\]
Then $[M_{0},M_{1}]$ is the $(\{S^0\}\cup\{S^1\})$-structure whose universe is the disjoint union of  the universes $M_{0}$ and $M_{1}$ (with the distance between elements of $M_0$ and elements of $M_1$ being~1) and such that
\begin{itemize}
\item
$P_k^M$ is the characteristic function of $M_k$ for $k=0,1$.
\item
For every $n$-ary predicate symbol $R$ of $S$ and every $\bar a\in M^n$,
\[
(R^{k})^{[M_{0},M_{1}]}(\bar a) =
\begin{cases}
 R^{M_{k}}(\bar a), \quad\text{if 
 $\bar a\in M_k^n$}\\
0, \quad\text{otherwise}.
\end{cases}
\]
\item
For every $n$-ary operation symbol $f$ of $S$ and every $\bar a\in M^n$,
\[
(f^{i})^{[M_{0},M_{1}]}(\bar a) =
\begin{cases}
 f^{M_{k}}(\bar a), \quad&\text{if 
 $\bar a\in M_k^n$}\\
a, &\text{otherwise},
\end{cases}
\]
where $a$ is a fixed element of $M$.
\end{itemize}

Now, the assumption that the logic $\mathcal{L}$ permits relativization to discrete predicates (see page~\pageref{L:conditions on the logic L}) allows us to fix for every $S$-sentence $\phi$ an $S^k$-sentence $\varphi^{k}$ such that $[M_{0},M_{1}] \models_{\mathcal{L}} \varphi^{k}$ if and only if $M_{k} \models_{\mathcal{L}} \varphi$.

\begin{proof}[Proof of the Main Theorem]
As observed above, we only have to show $\equiv_{\cl} \To \equiv_{\mathcal{L}}$ for signatures of cardinality less than $\kappa$.  Suppose that this is not the case, and fix a signature $\mathbf S$ of cardinality less than $\kappa$, and $\mathbf S$-structures $M_{0},M_{1}$ such that $M_{0}{\equiv}_{\cl} M_{1}$ and
\begin{equation*}
	M_{0} \models_{\mathcal{L}} \gamma \quad \text{but} \quad M_{1} \models_{\mathcal{L}} \gamma \le r \tag{\dag}
\end{equation*}
for some $\mathcal{L}$-sentence $\gamma$ and some $r\in\mathbb{Q}\cap(0,1)$.   
Our goal is to show that these structures can be taken to be metrically isomorphic; by (\dag), this would contradict property (3) of Definition~\ref{Definition:RealValuedLogic}.

Since $\mathcal{L}$ has the finite occurrence property, we may assume that the vocabulary $S$ is finite. Let $\{c_{i}\}_{i < \kappa}$ be a set of new constants and for $\textbf{i} = i_{1},\dots,i_{n} \in \kappa$, denote $c_{i_{1}},\dots,c_{i_{n}}$ by $c_{\textbf{i}}$.  For each $X \subseteq \kappa$, let $S_{X} = S^{0} \cup S^{1} \cup \{c_{i}^{0},c_{i}^{1}\}_{i \in X}$.  Define an $S_{\kappa}$-theory $T$ as follows:
\begin{multline*}
	T= \{\gamma^{0}\} \cup \{\gamma^{1} \le r\} \cup \{P_0(c_{i}^{0})\}_{i < \kappa} \cup \{P_1(c_{i}^{1})\}_{i < \kappa}\cup \\
	\big\{\, \psi^{0}(c_{\textbf{i}}^{0}) \lukimp (\psi^{1}(c_{\textbf{i}}^{1}) \geq s)  \ \big| \  
	\text{$\psi(\bar{x})$ an $S$-formula of $\cl$,}\\
	\text{$\textbf{i}$ in $\kappa$ with $\ell(\textbf{i})=\ell(\bar{x})$, $s\in\mathbb{Q}\cap(0,1)$} \,\big\}.
\end{multline*}
Our initial goal is to show that $T$ is consistent.  In order to do so, it is sufficient to show that the $S_{1}$-theory
\begin{multline*}
	T_{1} = \{\gamma^{0}\} \cup \{\gamma^{1} \le r\} \cup 
	\{P_0(c_{0}^{0})\} \cup \{P_1(c_{0}^{1})\}\cup \\
	\big\{\, \psi^{0}(c_{0}^{0}) \lukimp (\psi^{1}(c_{0}^{1}) \geq s) \ \big| \ 
	\text{$\psi(x)$ an $S$-formula of $\cl$, $s\in\mathbb{Q}\cap(0,1)$} \,\big\}.
\end{multline*}
is consistent, since any model $N$ of $T_1$ can be expanded to a model of $T$ by defining $(c_{i}^{0})^{N} = (c_{0}^{0})^{N}$ and  $(c_{i}^{1})^{N} = (c_{0}^{1})^{N}$ for $i < \kappa$. 

Now, $T_{1}$ is countable since $S$ is finite, so by the $\omega$-compactness of $\mathcal{L}$, we need only show that every finite subset of $T_{1}$ has a model. 
\begin{claim} 
 Let $\{\psi_{k}(x)\}_{k \leq m}$ be a finite set of $S$-formulas of $\cl$ and let $s\in\mathbb{Q}\cap(0,1)$ be given.  Then for every $a\in M_0$ there is $b\in M_1$ such that
 \[
[M_{0},M_{1}] \models_{\mathcal{L}} 
P_0[a] \land P_1[b]\land
\bigwedge_{k \leq m} ( \psi_{k}^{0}[a] \lukimp (\psi_{k}^{1}[b] \geq s) ).
\]
\end{claim}

\medskip\noindent\emph{Proof of the claim.}
Fix $a\in M_{0}$ and $s\in\mathbb{Q}\cap(0,1)$.
For each $k\le m$ choose $r_{k},t_{k}\in\mathbb{Q}\cap(0,1)$ such that 
\[
(\psi _{k}[a])^{M_0}-(1-s)\leq r_{k}<t_{k}\leq (\psi _{k}[a])^{M_0}
\]
and set $\epsilon =\min_{k\le m}t_{k}-r_{k}$. Since $M_{0}\models_{\mathcal{L}}  \exists
x(\bigwedge _{k\le m}(\psi _{k}(x)\geq t_{k}))$, the same sentence holds in $M_{1}$,
thus there is $b$ in $M_{1}$ such that $(\bigwedge _{k\le m}(\psi _{k}[b]\geq
t_{k}))^{M_{1}}\geq 1-\epsilon$. For each $k$,
\[
(\psi _{k}[b]\geq t_{k})^{M_{1}}\geq 1-\epsilon \geq 1-(t_{k}-r_{k}).
\]
Hence, by Proposition~\ref{Proposition:InequalitiesAndTheLukasiewiczImplication},
\[
(\psi _{k}[b])^{M_{1}}\geq t_{k}+1-(t_{k}-r_{k})-1=r_{k}\geq (\psi
_{k}[a])^{M_0}-(1-s)=s+(\psi _{k}[a])^{M_0}-1,
\]
which by the same proposition yields $\ (\psi _{k}[b]\geq s)^{M_{1}}\geq
(\psi _{k}[a])^{M_0}$. Thus, 
\[
[M_{0},M_{1}]\models_{\mathcal{L}}  \bigwedge _{k\le m}(\psi _{k}^{0}[a]\to_{L}(\psi _{k}[b]\geq s)).
\]

By the claim, every  finite subset of $T_{1}$ is satisfied by an expansion by constants of $[M_{0},M_{1}]$; hence, by the $\omega$-compactness of $\mathcal{L}$, the theory $T_{1}$ has a model. As observed above, every such model yields a model of $T$.

Our next (and final) goal is to show that $T$ has a model $[\hat{M}_{0},\hat{M}_{1}$] such that the set $\big\{ (c_{i}^{0})^{\hat{M}_{0}} \big\}_{i < \kappa}$ is dense in $\hat{M}_{0}$, and the set $\big\{ (c_{i}^{1})^{\hat{M}_{1}} \big\}_{i < \kappa}$ is dense in $\hat{M}_{1}$.  Once this is accomplished, the definition of $T$ ensures that the map
\[
	(c_{i}^{0})^{\hat{M}_{0}} \mapsto (c_{i}^{1})^{\hat{M}_{1}}
\]
is  a metric isomorphism between a dense subset of $\hat{M}_{0}$ and a dense subset of $\hat{M}_{1}$.  Since all the predicates in $S$ are uniformly continuous with respect to the distinguished metric, our isomorphism can be uniquely extended to a metric isomorphism between the completion of $\hat{M}_{0}$ and the completion of $\hat{M}_{1}$.  But since, by assumption, every structure is equivalent in $\mathcal L$ to its completion, the completion of  $\hat{M}_{0}$ satisfies $\gamma$, whereas the completion of $\hat{M}_{1}$ satisfies $\gamma \le r$, which contradicts the isomorphism property of $\mathcal L$ (see Definition~\ref{Definition:LogicalSystem}).

By the preceding remark, all that remains to show is that there is a model of $T$ that omits all the types
\[
	\Sigma_{\epsilon}(x) =
	 \big\{\, (P_0(x)\land d(x,c_{i}^{0}) \geq \epsilon) \, \vee \, 
	 (P_1(x)\land d(x,c_{j}^{1}) \geq  \epsilon )\, \big\}_{i,j < \kappa} \qquad (\epsilon\in\mathbb{Q}\cap(0,1)).
\]
Since $\mathcal{L}$ has the $\kappa$-Omitting Types Property, it suffices to show that $\Sigma_{\epsilon}(x)$ is not $\kappa$-principal over $T$, for each $\epsilon\in\mathbb{Q}\cap(0,1)$. Suppose that $\Sigma_{\epsilon}(x)$ is $\kappa$-principal for some $\epsilon$, and let $\Phi(x)$ generate $\Sigma_{\epsilon}(x)$ over $T$.

Fix $\delta\in\mathbb{Q}\cap(0,\epsilon)$. For $X \subseteq \kappa$ let us denote $T \upharpoonright S_{X}$ as $T_{X}$, and let
\[
	X_{0} = \{0\} \cup \big\{\, j \mid \text{$c_{j}^{0}$ or $c_{j}^{1}$ occurs in $\Phi(x)$} \,\big\}.
\]
Since $|\Phi(x)|< \kappa$ and $\mathcal{L}$ has the finite occurrence property, there is $j_{1} \in \kappa \setminus X_{0}$.  Let $X_{1} = X_{0} \cup \{j_{1}\}$.  We now use the $|\Phi|$-compactness of $\mathcal{L}$ to show that the set
\begin{equation*}
\tag{$\ddagger$}
	T_{X_{1}} \cup \Phi(x) \cup 
	\{\neg P_0(x)\lor d(x,c_{j_{1}}^{0}) \le \delta\} 
\end{equation*}
is satisfiable.  Since $\Phi(x)$ generates $\Sigma_{\epsilon}(x)$, there is an $S_{\kappa}$-structure $[M_{0}',M_{1}']$ and $a \in M_{0}'$ such that $[M_{0}',M_{1}'] \models_{\mathcal{L}} T \cup \Phi[a]$.  If 
$a\in M_1'$, the satisfiablity of ($\ddagger$) is immediate (since $\neg P_0^{[M_{0}',M_{1}']}(a)=1$), so suppose $a\in M_0'$.  Since $M_{0}{\equiv}_{\cl} M_{1}$, the argument used to prove our claim shows that
\[
	\begin{matrix}
		\vspace{\stretch{1}}  (\diamond)\vspace{\stretch{1}}

		&

		\begin{minipage}{10cm}
		for every finite set $\{\psi_{k}(\bar{x},y)\}_{k \leq m}$ of $S$-formulas of $\cl$ and every $s\in\mathbb{Q}\cap(0,1)$ there is $b$ in $M_{1}'$ such that whenever $\ell(\textbf{i}) = \ell(\bar{x})$, the structure $[M_{0}',M_{1}']$ satisfies
\[
	\bigwedge_{k \leq m} ( \psi_{k}^{0}(c_{\textbf{i}}^{0},y)[a] \lukimp ( \psi_{k}^{1}(c_{\textbf{i}}^{1},y)[b] \ge s)).
\]
		\end{minipage}
	\end{matrix}
\]

Let $\Gamma(c_{\textbf{i}}^{0},c_{\textbf{i}}^{1},c_{j_{1}}^{0},c_{j_{1}}^{1})$ be a finite subset of $T_{X_{1}}$, where $\textbf{i} \in X_{0}$ and all the new constants are being displayed, and let $S(\Gamma)$ denote the finite part of $S_{X_{0}}$ that occurs in $\Gamma$.  Notice that the reduct $[M_{0}',M_{1}'] \upharpoonright S(\Gamma)$ satisfies $\Gamma \upharpoonright S_{X_{0}}$.  Let
\[
\psi_{0}(\bar{x},y),\dots, \psi_{m}(\bar{x},y)
\]
be a list of all the $S$-formulas such that the implications
\[
\psi_{0}^{0}(c_{\textbf{i}}^{0},c_{j_{1}}^{0}) \lukimp (\psi_{0}^{1}(c_{\textbf{i}}^{1},c_{j_{1}}^{1})\ge s_0),\dots,
\psi_{m}^{0}(c_{\textbf{i}}^{0},c_{j_{1}}^{0}) \lukimp (\psi_{m}^{1}(c_{\textbf{i}}^{1},c_{j_{1}}^{1})\ge s_m)
\]
occur in $\Gamma$, let $s=\min_{k\le m}s_k$, and  fix $b\in M_{1}'$ corresponding to  $\{\psi_{k}(\bar{x},y)\}_{k \leq m}$ and $s$ as given by $(\diamond)$. Now let
\[
( [M_{0}',M_{1}'] \upharpoonright S(\Gamma), a, b)
\]
denote the expansion of $[M_{0}',M_{1}'] \upharpoonright S(\Gamma)$ to $S_{X_1}$ where $a$ is the interpretation of $c_{j_{1}}^{0}$ and $b$ is interpretation of $c_{j_{1}}^{1}$. Then, by $(\diamond)$, we have
\[
	\big( [M_{0}',M_{1}'] \upharpoonright S(\Gamma), a, b \big) \models_{\mathcal{L}} \Gamma(c_{\textbf{i}}^{0},c_{\textbf{i}}^{1},c_{j_{1}}^{0},c_{j_{1}}^{1}). 
\]
By the choice of $a$,  $[M_{0}',M_{1}'] \models_{\mathcal{L}} T \cup \Phi[a]$, and trivially, we also have
\[
( [M_{0}',M_{1}'] \upharpoonright S(\Gamma), a, b)\models_{\mathcal{L}} d(x,c_{j_{1}}^{0})[a]\le\delta.
\]
Therefore $a$ realizes
\[
\Gamma(c_{\textbf{i}}^{0},c_{\textbf{i}}^{1},c_{j_{1}}^{0},c_{j_{1}}^{1}) \cup \Phi[a] \cup 
	\{\neg P_0(x)\lor d(x,c_{j_{1}}^{0}) \le \delta\}
\]
in the structure $( [M_{0}',M_{1}'] \upharpoonright S(\Gamma), a, b \big)$. Since $\mathcal{L}$ is $|\Phi|$-compact and $\Gamma$ is an arbitrary finite subset of $T_{X_{1}}$, this shows that ($\ddagger$) is satisfiable.

Fix now $j_{2} \in \kappa \setminus X_{1}$, and let $X_{2} = X_{1} \cup \{j_{2}\}$.  An argument symmetric to that which produced a model of ($\ddagger$) shows that the theory
\begin{equation*}
	T_{X_{2}} \cup \Phi(x) \cup
	 \{\, (\neg P_0(x)\lor d(x,c_{j_{1}}^{0}) \le  \delta)
	 \land  (\neg P_1(x)\lor d(x,c_{j_{1}}^{1}) \le \delta) \,\}
\end{equation*}
is satisfied by an $S_{X_{2}}$-structure.  To conclude the proof, we only need to expand this model to an $S_{\kappa}$-structure, i.e., we need to find interpretations for the constants $c_{i}^{0},c_{i}^{1}$ with $i \in \kappa \setminus X_{2}$ in such a way that
\begin{equation*}
	T \cup \Phi(x) \cup 
	\{\, (\neg P_0(x)\lor d(x,c_{j_{1}}^{0}) \le \delta)
	 \land  (\neg P_1(x)\lor d(x,c_{j_{1}}^{1}) \le \delta) \,\}
\end{equation*}
is still satisfied; but this can be done by simply giving $c_{i}^{0},c_{i}^{1}$, for $i \in \kappa \setminus X_{2}$, the same interpretation as $c_{0}^{0},c_{0}^{1}$. Since $\delta<\epsilon$, 
\[
T \cup \Phi(x) \not\models_{\mathcal L}\Sigma_{\epsilon}(x),
\]
so $\Phi(x)$ does not generate $\Sigma_{\epsilon}(x)$, as presumed. This concludes the proof that $\Sigma_{\epsilon}(x)$ is not $\kappa$-principal over $T$, and thus the proof of the Main Theorem.
\end{proof}

The preceding proof is a refinement of the proof of the main theorem in \cite{Lindstrom:1978}.

\begin{remark}
The $\kappa$-Omitting Types Property for a theory $T$ states that for every set of at most $\kappa$ types that are not $\kappa$-principal over $T$ there is a model of $T$ that omits all the types in the set.  In the proof of Theorem~\ref{Theorem:TheMainTheorem}, we needed only a weak version of this property, namely, we need the existence of a model of $T$ that omits countable sets of types that are not $\kappa$-principal.  Thus, in basic continuous logic, the $\kappa$-Omitting Types Property is equivalent to this apparently weaker version of it.
\end{remark}

\begin{corollary} \label{C: maximality}
Let $\mathcal{L}$ be a $[0,1]$-valued logic. 
\begin{enumerate}
\renewcommand{\theenumi}{\textup{(\arabic{enumi})}}
\renewcommand{\labelenumi}{\theenumi}
\item
If $\mathcal L$  satisfies properties  \textup{(1)--(5)} of page~\pageref{L:conditions on the logic L},  $\mathcal L$ extends  {\L}ukasziewicz-Pavelka logic, and there exists an uncountable regular cardinal $\kappa$ such that $\mathcal L$ satisfies the $\kappa$-Omitting Types Property, then $\mathcal{L}$ is equivalent to {\L}ukasziewicz-Pavelka logic for  signatures of cardinality less than $ \kappa$.
\item
If $\mathcal L$  satisfies properties  \textup{(1)--(4)} of page~\pageref{L:conditions on the logic L},  $\mathcal L$ extends continuous logic, and there exists an uncountable regular cardinal $\kappa$ such that $\mathcal L$ satisfies the $\kappa$-Omitting Types Property for complete structures, then $\mathcal{L}$ is equivalent to continuous logic for  signatures of cardinality less than $\kappa$.
\end{enumerate}
\end{corollary}

\begin{proof}
We note first that the methods used in the proof of Theorem~\ref{Theorem:TheMainTheorem} to produce new structures from old ones (i.e., the construction of $M$ from $\{M_j\}_{j<\lambda}$ in the proof of~Proposition~\ref{Proposition:LambdaCompactness} and the construction of $[M_{0},M_{1}]$ from $M_{0}$ and $M_{1}$ on page~\pageref{P:amalgamation}) yield complete structures from complete structures and 1-Lipschitz structures from 1-Lipschitz structures. Hence,  (1) and (2) above follow by assuming throughout the proof  of Theorem~\ref{Theorem:TheMainTheorem} that all the structures involved are complete, or 1-Lipschitz, accordingly.
\end{proof}


\begin{thebibliography}{BYBHU08}

\bibitem[BF85]{Barwise-Feferman:1985}
J.~Barwise and S.~Feferman (eds.), \emph{Model-theoretic logics},
  Springer-Verlag, New York, 1985. \MR{87g:03033}

\bibitem[BYBHU08]{Ben-Yaacov-Berenstein-Henson-Usvyatsov:2008}
Ita{\"{\i}} Ben~Yaacov, Alexander Berenstein, C.~Ward Henson, and Alexander
  Usvyatsov, \emph{Model theory for metric structures}, Model theory with
  applications to algebra and analysis. {V}ol. 2, London Math. Soc. Lecture
  Note Ser., vol. 350, Cambridge Univ. Press, Cambridge, 2008, pp.~315--427.
  \MR{MR2436146}

\bibitem[BYU07]{Ben-Yaacov-Usvyatsov:2007}
Ita{\"{\i}} Ben~Yaacov and Alexander Usvyatsov, \emph{On {$d$}-finiteness in
  continuous structures}, Fund. Math. \textbf{194} (2007), no.~1, 67--88.
  \MR{2291717 (2007m:03080)}

\bibitem[BYU10]{Ben-Yaacov-Usvyatsov:2010}
\bysame, \emph{Continuous first order logic and local stability}, Trans. Amer.
  Math. Soc. \textbf{362} (2010), no.~10, 5213--5259. \MR{2657678}

\bibitem[Cai]{Caicedo:201?}
Xavier Caicedo, \emph{A {L}indstr\"{o}m's theorem for {L}ukasiewicz and
  continuous logic}, Manuscript.

\bibitem[Cai93]{Caicedo:1993}
\bysame, \emph{Compactness and normality in abstract logics}, Ann. Pure Appl.
  Logic \textbf{59} (1993), no.~1, 33--43. \MR{1197204 (93m:03062)}

\bibitem[Cai95]{Caicedo:1995}
\bysame, \emph{Continuous operations on spaces of structures}, Quantifiers:
  Logics, Models and Computation I, Synthese Library, vol. 248, 1995,
  pp.~263--296.

\bibitem[Cai99]{Caicedo:1999}
\bysame, \emph{The abstract compactness theorem revisited}, Logic and
  Foundations of Mathematics, Synthese Library, vol. 280, 1999, pp.~131--141.

\bibitem[CK66]{Chang-Keisler:1966}
Chen-chung Chang and H.~Jerome Keisler, \emph{Continuous model theory}, Annals
  of Mathematics Studies, No. 58, Princeton Univ. Press, Princeton, N.J., 1966.
  \MR{MR0231708 (38 \#36)}

\bibitem[GJ76]{Gillman-Jerison:1976}
Leonard Gillman and Meyer Jerison, \emph{Rings of continuous functions},
  Springer-Verlag, New York, 1976, Reprint of the 1960 edition, Graduate Texts
  in Mathematics, No. 43. \MR{0407579 (53 \#11352)}

\bibitem[GM04]{Garcia-Matos:2004}
Marta Garc{\'{\i}}a-Matos, \emph{A framework for maximality and interpolation
  in abstract logics with and without negation}, Aspects of universal logic,
  Travaux Log., vol.~17, Univ. Neuch\^atel, Neuch\^atel, 2004, pp.~66--86.
  \MR{2168187}

\bibitem[GMV05]{Garcia-Matos-Vaananen:2005}
Marta Garc{\'{\i}}a-Matos and Jouko V{\"a}{\"a}n{\"a}nen, \emph{Abstract model
  theory as a framework for universal logic}, Logica universalis, Birkh\"auser,
  Basel, 2005, pp.~19--33. \MR{MR2134728 (2005k:03091)}

\bibitem[H{\'a}j97]{Hajek:1997}
Petr H{\'a}jek, \emph{Fuzzy logic and arithmetical hierarchy. {II}}, Studia
  Logica \textbf{58} (1997), no.~1, 129--141. \MR{1432158 (98j:03038)}

\bibitem[H{\'a}j98]{Hajek:1998}
\bysame, \emph{Metamathematics of fuzzy logic}, Trends in Logic---Studia Logica
  Library, vol.~4, Kluwer Academic Publishers, Dordrecht, 1998. \MR{1900263
  (2003c:03048)}

\bibitem[HI02]{Henson-Iovino:2002}
C.~Ward Henson and Jos{\'e} Iovino, \emph{Ultraproducts in analysis}, Analysis
  and logic (Mons, 1997), London Math. Soc. Lecture Note Ser., vol. 262,
  Cambridge Univ. Press, Cambridge, 2002, pp.~1--110. \MR{1 967 834}

\bibitem[HPS00]{Hajek-Paris-Stepherson:2000}
Petr H{\'a}jek, Jeff Paris, and John Shepherdson, \emph{Rational {P}avelka
  predicate logic is a conservative extension of {\l}ukasiewicz predicate
  logic}, J. Symbolic Logic \textbf{65} (2000), no.~2, 669--682. \MR{1771076
  (2001m:03048)}

\bibitem[Iov01]{Iovino:2001}
Jos\'e Iovino, \emph{On the maximality of logics with approximations.}, J.
  Symbolic Logic \textbf{66} (2001), no.~4, 1909--1918.

\bibitem[Lin69]{Lindstrom:1969}
P.~Lindstr{\"o}m, \emph{On extensions of elementary logic}, Theoria \textbf{35}
  (1969), 1--11. \MR{39 \#5330}

\bibitem[Lin78]{Lindstrom:1978}
Per Lindstr{\"o}m, \emph{Omitting uncountable types and extensions of
  elementary logic}, Theoria \textbf{44} (1978), no.~3, 152--156. \MR{586927
  (82d:03062)}

\bibitem[MN06]{Murinova-Novak:2006}
Petra Murinov{\'a} and Vil{\'e}m Nov{\'a}k, \emph{Omitting types in fuzzy logic
  with evaluated syntax}, MLQ Math. Log. Q. \textbf{52} (2006), no.~3,
  259--268. \MR{2239466 (2007f:03032)}

\bibitem[Nov89]{Novak:1989}
Vil{\'e}m Nov{\'a}k, \emph{Fuzzy sets and their applications}, Adam Hilger
  Ltd., Bristol, 1989, Translated from the Czech. \MR{1019090 (91a:04007)}

\bibitem[Nov90]{Novak:1990}
\bysame, \emph{On the syntactico-semantical completeness of first-order fuzzy
  logic. {II}. {M}ain results}, Kybernetika (Prague) \textbf{26} (1990), no.~2,
  134--154. \MR{1059796 (91f:03047b)}

\bibitem[Nov95]{Novak:1995}
\bysame, \emph{A new proof of completeness of fuzzy logic and some conclusions
  for approximate reasoning}, Proc. Int. Conference FUZZ-IEEE/FES'95
  (Yokohama), 1995.

\bibitem[Pav79a]{Pavelka:1979I}
Jan Pavelka, \emph{On fuzzy logic. {I}}, Z. Math. Logik Grundlag. Math.
  \textbf{25} (1979), no.~1, 45--52, Many-valued rules of inference. \MR{524558
  (80j:03038a)}

\bibitem[Pav79b]{Pavelka:1979II}
\bysame, \emph{On fuzzy logic. {II}. {E}nriched residuated lattices and
  semantics of propositional calculi}, Z. Math. Logik Grundlag. Math.
  \textbf{25} (1979), no.~2, 119--134. \MR{527904 (80j:03038b)}

\bibitem[Pav79c]{Pavelka:1979III}
\bysame, \emph{On fuzzy logic. {III}. {S}emantical completeness of some
  many-valued propositional calculi}, Z. Math. Logik Grundlag. Math.
  \textbf{25} (1979), no.~5, 447--464. \MR{543638 (80j:03038c)}

\end{thebibliography}

\providecommand{\bysame}{\leavevmode\hbox to3em{\hrulefill}\thinspace}
\providecommand{\MR}{\relax\ifhmode\unskip\space\fi MR }
\providecommand{\MRhref}[2]{%
  \href{http://www.ams.org/mathscinet-getitem?mr=#1}{#2}
}
\providecommand{\href}[2]{#2}

\end{document}